\documentclass[leqno]{amsart}
\usepackage{amsfonts,amssymb,amsmath,amsgen,amsthm}
\usepackage{hyperref}

\theoremstyle{plain}
\newtheorem{theorem}{Theorem}[section]
\newtheorem{definition}[theorem]{Definition}
\newtheorem{assumption}[theorem]{Assumption}
\newtheorem{lemma}[theorem]{Lemma}

\newtheorem{proposition}[theorem]{Proposition}
\theoremstyle{remark}
\newtheorem{remark}[theorem]{Remark}

\newtheorem{example}[theorem]{Example}

\def\k{{\kappa}}

\def\C{{\mathbf C}}
\def\R{{\mathbf R}}
\def\N{{\mathbf N}}
\def\Z{{\mathbf Z}}
\def\T{{\mathbf T}}
\def\Sph{{\mathbf S}} 
\def\Sch{{\mathcal S}}
\def\F{\mathcal F}
\def\O{\mathcal O}

\def\virgp{\raise 2pt\hbox{,}}

\def\({\left(}
\def\){\right)}
\def\<{\left\langle}
\def\>{\right\rangle}
\def\le{\leqslant}
\def\ge{\geqslant}

\def\1{{\bf 1}}

\def\Eq#1#2{\mathop{\sim}\limits_{#1\rightarrow#2}}
\def\Tend#1#2{\mathop{\longrightarrow}\limits_{#1\rightarrow#2}}

\def\d{{\partial}}
\def\l{\lambda}
\def\om{\omega}

\def\si{{\sigma}}
\def\eps{\varepsilon}
\DeclareMathOperator{\RE}{Re}

\numberwithin{equation}{section}

\begin{document}

\title[Geometric optics and instability for NLS and
Davey-Stewartson]{Geometric optics and  
instability for NLS and Davey-Stewartson models} 

\author[R. Carles]{R\'emi Carles}
\address[R. Carles]{Univ. Montpellier~2\\Math\'ematiques \\
  CC~051\\F-34095 Montpellier} 
\address{CNRS, UMR 5149\\  F-34095 Montpellier\\ France}
\email{Remi.Carles@math.cnrs.fr}

\author[E.~Dumas]{Eric Dumas}
\address[E.~Dumas]{Univ. Grenoble 1\\ Institut Fourier\\ 100, rue des
  Math\'ematiques-BP~74\\ 38402 Saint Martin d'H\`eres cedex\\ France} 
\email{Eric.Dumas@ujf-grenoble.fr}

\author[C. Sparber]{Christof Sparber}
\address[C. Sparber]{Department of Applied Mathematics and Theoretical
  Physics\\ 
  CMS, Wilberforce Road\\ Cambridge CB3 0WA\\ England}
\email{c.sparber@damtp.cam.ac.uk}

\thanks{This work was supported by the French ANR project
  R.A.S. (ANR-08-JCJC-0124-01) and by the Royal Society Research
  fellowship of C. Sparber}   
  
\begin{abstract} 
We study the interaction of (slowly modulated) high frequency waves
for multi-dimensional nonlinear  Schr\"odinger equations with gauge
invariant power-law  
nonlinearities and non-local perturbations. The model 
includes the 
Davey--Stewartson system in its elliptic-elliptic and
hyperbolic-elliptic variant.  
Our analysis reveals a new localization phenomenon for non-local
perturbations in the high frequency regime and allows us to infer  
strong instability results on the Cauchy problem in negative order Sobolev
spaces, where we prove norm inflation with infinite loss of
regularity by a constructive approach. 
\end{abstract}
\maketitle
\tableofcontents

\section{Introduction}
\label{sec:intro}

\subsection{Motivation}
\label{sec:motiv}

The \emph{Davey-Stewartson system} (DS) provides a canonical
description of the dynamics of weakly nonlinear two-dimensional waves
interacting with  
a mean-field $\chi (t,x_1,x_2) \in \R$; see \cite{Sulem} for more
details. In the following we shall consider  
\begin{equation}\label{eq:DS0}
\left \{
\begin{aligned}
  i\d_t \psi + \frac{1}{2}\(\eta \d^2_{x_1} + \d^2_{x_2} \) \psi &= 
   \(   \d_{x_1} \chi  
  + \mu |\psi|^2 \) \psi\ , \\
  \( \d^2_{x_1} + \d^2_{x_2} \)\chi & =      \lambda \d_{x_1} |\psi|^2, 
  \end{aligned}
   \right. \tag{DS} 
\end{equation}
where $(x_1,x_2) \equiv x \in \R^2$, $t\in \R$, and $\lambda,\mu\in
\R$ are some given parameters. 
In addition, the choice $\eta = \pm 1$ distinguishes between to the
so-called \emph{elliptic-elliptic} and the \emph{hyperbolic-elliptic}
variants of the DS system (see \cite{Sulem}).  
Clearly, the DS system with $\eta = +1$ and $\lambda = 0$ simplifies
to the cubic \emph{nonlinear Schr\"odinger equation} (NLS), which we
consider more generally in the $d$-dimensional case
\begin{equation*}
  i\d_t \psi + \frac{1}{2}\Delta \psi =
   \mu |\psi|^2\psi\ ,\quad  x \in \R^d. 
\end{equation*}
The cubic NLS equation is a canonical model for weakly nonlinear wave
propagation in dispersive media and has numerous applications in e.g.  
nonlinear optics, quantum superfluids, or the description of water
waves, cf. \cite{Sulem}. We shall allow for more
general, gauge invariant, nonlinearities and consider  
\begin{equation}\label{eq:NLS}
  i\d_t \psi + \frac{1}{2}\Delta \psi =
   \mu |\psi|^{2\nu}\psi\ , \quad x \in \R^d \tag{NLS},
\end{equation}
where $\nu \in \N^\star$. For such equations, one usually distinguishes
between \emph{focusing} $\mu <0$ and \emph{defocusing} $\mu >0$
nonlinearities. The sign of $\mu$ has a huge 
impact on the issue of global well-posedness, since it is well known
(cf. \cite{Sulem} for a broad review)  
that for $\mu < 0$ \emph{finite-time blow-up} of solutions may occur
for $d\ge 2$, that is:   
$$
 \lim_{t \to T^*} \| \nabla \psi (t,\cdot) \|_{L^2(\R^d)} = \infty
 \, , \quad T^* < +\infty. 
$$ 
Thus, in general we cannot expect global well-posedness to hold in,
say, $H^1(\R^d)$.  
On the other hand, one might ask about the possibility that even local
(in time) well-posedness fails. To be more precise, we recall the
following definition: 
\begin{definition}\label{def:WP}
Let $\si, s \in \R$. The Cauchy problem for \eqref{eq:genNLS}
is well posed from $H^{s}(\R^d)$ to $H^\si(\R^d)$ if, for all bounded
subset $B\subset H^{s}(\R^d)$, there exist  
$T>0$ and a Banach space $X_T\hookrightarrow C([0,T];H^\si(\R^n))$
such that: 
\begin{enumerate}
\item For all $\varphi\in B \cap H^\infty$, \eqref{eq:genNLS}  
has a unique solution $\psi \in X_T$ with $\psi_{\mid t=0}=\varphi$. 
\item The mapping $\varphi\in (B \cap H^\infty,\| \cdot \|_{H^s})
  \mapsto \psi\in C([0,T];H^\si(\R^n))$ is continuous.
  \end{enumerate}
\end{definition}
The negation of the above definition is called a \emph{lack of
  well-posedness} or \emph{instability}. 
In order to gain a rough idea why instability occurs, we consider the
Cauchy problem of \eqref{eq:NLS} with initial data $\psi_0 \in
H^s(\R^d)$. Under the assumption $\nu \in \N^\star$, the nonlinearity
is smooth, and thus,  local well-posedness (from $H^s(\R^d)$ to
$H^s(\R^d)$) holds  for sufficiently large $s$ ($s>d/2$ does the
job). On the other hand, one should note that \eqref{eq:NLS} is
invariant under \emph{Galilei transformations},  
$$
\psi(t,x) \mapsto e^{iv\cdot x -i |v|^2 t/2}\psi(t,x-vt), \quad v \in \R^d,
$$
which leave the $L^2(\R^d)$ norm invariant. In addition, solutions to
\eqref{eq:NLS} are  invariant under the \emph{scaling symmetry} 
\begin{equation*}
  \psi(t,x) \mapsto \Lambda ^{-1/ \nu} \psi \( \frac{t}{\Lambda^2},
\frac{x}{\Lambda} \), \quad \Lambda >0. 
\end{equation*}
Denoting 
\begin{equation}
  \label{eq:sc}
  s_c : = \frac{d}{2} - \frac{1}{\nu},
\end{equation}
this scaling is easily seen to leave the homogeneous Sobolev space
$\dot H^{s_c}(\R^d)$ invariant and thus we heuristically expect local
well-posedness to hold only in $H^s(\R^d)$ with $s\ge \max \{s_c,
0\}$.   
The reason for this being that for $s <\max \{s_c, 0\}$ and
sufficiently large $\Lambda >0$ we can use the scaling symmetry of 
(NLS) to relate the norm of large solutions at time $t>0$ 
to the norm of small solutions at some time $t^* < t$. In other words, 
the difference between two solutions will immediately become very big
in $H^s(\R^d)$,  
even if they are close to each other initially. 

For the cubic NLS equation we have $s_c = 0$ if $d=2$, and thus
instability should occur for $\psi_0\in H^s(\R^2)$ with
$s<0$. Moreover, we expect 
the same behavior to be true also for the DS system, since the latter
can be written in form of an NLS with non-local perturbation, i.e. 
\begin{equation}\label{eq:DS}
  i\d_t \psi + \frac{1}{2}\(\eta  \d^2_{x_1} + \d^2_{x_2} \) \psi= \lambda
  E \(|\psi|^2\)\psi 
  + \mu |\psi|^2\psi\ , 
  \quad x \in \R^2.
\end{equation}
Here the operator $E$ acting as a Fourier
multiplier on $|\psi|^2$ is defined via 
\begin{equation}\label{eq:E}
\widehat {E (f)}(\xi) = \frac{\xi_1^2}{\xi_1^2+\xi_2^2} \, \widehat f(\xi),
\end{equation}
where $(\xi_1,\xi_2)=\xi\in \R^2$ and $\widehat  f$ denotes the
Fourier transform of  $f$, defined as
\begin{equation}\label{eq:Fourier}
 (\F f)(\xi)\equiv \widehat
 f(\xi)=\frac{1}{(2\pi)^{d/2}}\int_{\R^d}f(x)e^{-i x\cdot \xi} dx.
\end{equation}
With this normalization, we have $\F^{-1}g = \F \check g$, with 
$\check g = g(-\cdot)$. Note that the non-local term in \eqref{eq:DS} scales like the nonlinearity in the cubic NLS equation, since the kernel 
\begin{equation}\label{eq:K} 
\widehat K(\xi) = \frac{\xi_1^2}{\xi_1^2+\xi_2^2} \in L^\infty(\R^2),
\end{equation}
is \emph{homogeneous of degree zero}. We therefore expect instability
of the DS system in Sobolev spaces of negative order.  
It will be one of the main tasks of
this work to rigorously prove this type of instability, which can be
seen as a negative result, complementing the 
well-posedness theorems of \cite{GhSa90} (see also \cite{GaZh08}).
To this end, we shall rely on the framework of \emph{weakly nonlinear
  geometric optics} (WNLGO), developed in \cite{CDS10} for NLS. We shall
henceforth study, as a first step, the interaction of \emph{highly
  oscillatory waves}   
within \eqref{eq:DS} and describe the possible (nonlinear)
\emph{resonances} between them. In our opinion this is interesting in
itself since  
it generalizes the results of \cite{CDS10} and reveals a new
localization property for non-local operators in the high frequency
regime. Moreover, we shall see that possible resonances heavily depend
on the choice of $\eta =\pm 1$.  
\smallbreak

In order to treat the DS system and the NLS equation simultaneously,
we shall from now on consider the following NLS type model 
\begin{equation}
\label{eq:genNLS}
i\d_t \psi + \frac{1}{2} \Delta_\eta \psi =  
\lambda E \(|\psi |^{2\nu}\) \psi + \mu |\psi |^{2\nu} \psi \quad ,
\quad \psi(0,x) = \psi_0(x), 
\end{equation}
with $x\in \R^d$, $\l,\mu\in \R$, and $\nu\in\N^\star$ and a
generalized dispersion of the form 
\begin{equation}
\label{eq:etaLapl}
\Delta_\eta := \sum_{j=1}^d \eta_j \partial_{x_j}^2 , 
\quad \eta_j = \pm 1 ,
\end{equation} 
Furthermore, we generalize the operator $E$ given in \eqref{eq:E} by imposing the following assumption:
\begin{assumption}\label{ass:E} 
The operator $E$ is given by 
\begin{equation*}
  E(f)=K\ast f,\quad K\in \Sch'(\R^d), 
\end{equation*} 
where $\widehat K$ is homogeneous of degree zero, and continuous away from
the origin. 
\end{assumption} 

\begin{remark} For $\lambda =0$ and 
nonuniform signs of the $\eta_j$'s, equation \eqref{eq:genNLS} simplifies to the 
so-called \emph{hyperbolic NLS}, which arises for example in the
description of surface-gravity waves on deep water, cf. \cite{Sulem}.  
\end{remark}

\subsection{Weakly nonlinear geometric optics}
\label{sec:main}  

We aim to understand the interaction of high frequency waves 
within solutions to \eqref{eq:genNLS}. To this end, we consider the
following \emph{semi-classically scaled} model 
\begin{equation}
\label{eq:genWNLGO}
i\eps \d_t u^\eps + \frac{\eps^2}{2} \Delta_\eta u^\eps =  
\eps \lambda E \(|u^\eps |^{2\nu}\) u^\eps 
+ \eps  \mu |u^\eps |^{2\nu} u^\eps \quad , \quad u^\eps(0,x) = u^\eps_0(x),
\end{equation}
where $0 < \eps \ll 1$ denotes a small semi-classical parameter. 
The singular limiting regime where $\eps \to 0$ yields the high
frequency asymptotics for \eqref{eq:genNLS} in a \emph{weakly}
nonlinear scaling   
(note that $\eps$ appears in front of the nonlinearities). The latter is
known to be critical as far as geometric optics is concerned, see
e.g. \cite{CaBook}. 

As in \cite{CDS10}, we shall assume that  
\eqref{eq:genWNLGO} is subject to initial data given by a 
superposition of $\eps$-oscillatory plane waves, i.e.
\begin{equation}
  \label{eq:genWNLGOini}
u^\eps_0(x) = \sum_{j\in J_0}\alpha_j(x)e^{i\k_j\cdot x/\eps}, 
\end{equation}
where for some index set $J_0\subseteq \N$ we are given initial wave
vectors $\k_j \in \R^d$ with corresponding smooth, rapidly decaying
amplitudes $\alpha_j \in \Sch(\R^d,\C)$. Since, in general, we can
allow for countable many $\alpha_j$'s,   
we shall impose the following summability condition:
\begin{assumption} \label{hyp:sum} The initial amplitudes satisfy
\begin{equation*}
  \sum_{j \in J_0 } \<\k_j\>^2 \|\widehat   \alpha _j \|_{L^1\cap L^2} +
\sum_{j \in J_0 } \|\widehat {\Delta \alpha _j} \|_{L^1\cap L^2}  <+ \infty,
\end{equation*}
where $\<\k_j\>:= (1+ |\k_j|^2)^{1/2}$.
\end{assumption}
This condition will become clear in Section~\ref{sec:WNLGO}, where we
justify multiphase weakly nonlinear geometric optics using the
framework of Wiener algebras.
\smallbreak

The initial condition \eqref{eq:genWNLGOini} induces high frequency
oscillations within the solution of \eqref{eq:genWNLGO}.  
The first main result of this work concerns the approximation of the
exact solution $u^\eps$ of \eqref{eq:genWNLGO} by (possibly 
countably many) slowly modulated plane waves. 

\begin{theorem}\label{theo:WNLGO}
Let $d\ge1$, $\lambda , \mu \in \R$ and $\nu\in\N^*$, and let $E$ satisfy Assumption \ref{ass:E}. Consider initial data of the form \eqref{eq:genWNLGOini} 
with $\k_j\in \Z^d$ and $\alpha_j \in \mathcal S(\R^d,\C)$ satisfying
Assumption~\ref{hyp:sum}.   

Then there exist $T>0$, and $C,\eps_0>0$, 
such that for all $\eps \in ]0,\eps_0]$, there exists a unique solution 
$u^\eps\in C([0,T];L^\infty\cap L^2)$ to 
\eqref{eq:genWNLGO}--\eqref{eq:genWNLGOini}. It can be approximated by 
\begin{equation*}
\begin{aligned}
\sup_{t\in [0,T]}
\left\lVert u^\eps(t,\cdot)-u_{\rm app}^\eps(t,\cdot)
\right\rVert_{L^\infty\cap L^2(\R^d)} 
& \Tend \eps 0 0  \quad \text{if } \lambda\neq0, \\
\sup_{t\in [0,T]}
\left\lVert u^\eps(t,\cdot)-u_{\rm app}^\eps(t,\cdot)
\right\rVert_{L^\infty\cap L^2(\R^d)} & \le C\eps \quad \text{if } \lambda=0.
\end{aligned}
\end{equation*}
Here, the approximate solution $u_{\rm app}^\eps\in
C([0,T];L^\infty\cap L^2)$ is given by
\begin{equation*}
u_{\rm app}^\eps(t,x) = \sum_{j\in J} a_j(t,x) 
e^{i \phi_j (t,x) / \eps},
\end{equation*}
where the amplitudes $a_j \in C([0,T]; L^\infty\cap  L^2(\R^d))$ solve
the system  
\eqref{eq:transportsystemgen} and the phases $\phi_j$ are given by 
\begin{equation*}
\phi_j (t,x) = \k_j \cdot x 
- \frac{t}{2} \sum_{\ell=1}^d \eta_\ell \k_{j,\ell}^2 . 
\end{equation*}
In addition the index set $J \subseteq \N$ can be determined  from
$J_0$ by following the approach outlined in
Sections~\ref{sec:NLSphases} and \ref{sec:phasesDS}.
\end{theorem}
\begin{remark}
  The assumption $\k_j\in \Z^d$ is introduced to avoid small divisors
  problems. This aspect is discussed in more details in
  \cite{CDS10}. Following the strategy of \cite{CDS10}, we could state
  a more general result here. We have chosen not to do so, for
  the sake of readability. 
\end{remark}
In general we have $J_0 \subseteq J$, due to possible resonances, i.e. the creation of new (characteristic) oscillatory phases
$\phi_j$ not originating 
from the given initial data but solely due to nonlinear interactions.  
The above theorem generalizes the results of \cite{CDS10}, 
in three different directions: 
\begin{enumerate}

\item The approximation result is
extended to $L^2\cap L^\infty$ (in \cite{CDS10} we only proved an
approximation in $L^\infty$).

\item We allow for non-elliptic Schr\"odinger operators 
corresponding to non-uniform sign for the $\eta_j$'s.  
This yields a resonance structure which is different from
the elliptic case (see Section \ref{sec:DSWNLGO}). In particular, one
should note that for   
$d=2$, $\eta_1 =-\eta_2$ and $\k_j=(k,k) \in \R^2$, the 
corresponding phase $\phi_j$ simplifies to $\phi_j(x) = k (x_1 + x_2)$, 
describing $\eps$-oscillations which do \emph{not} propagate 
in time.

\item In comparison to \cite{CDS10} we can now also
take into account 
the non-local Fourier multiplier $E$. This
operator, roughly speaking, behaves like local nonlinearity in the
limit $\eps \to 0$ (see Section \ref{sec:nonloc}).  
The behavior is therefore qualitatively different from earlier results
given in \cite{GMS-p}, where it has been proved that for (slightly more
regular) integral kernels $K$, such that $\langle \xi\rangle \widehat
K(\xi)  \in L^\infty$, no   
new resonant phase can be created by $E$, in contrast to our
work. Notice that the present work shows as a by-product
that the same conclusion also holds for the Schr\"odinger--Poisson system (on
$\R^d$, $d\ge 3$, so $\widehat K(\xi)=c_d/|\xi|^2$), even though in
that case, the kernel  
does not satisfy the above assumption; see Remark~\ref{rem:SP} below
for more details.   
\end{enumerate}

\begin{remark} Finally, we underscore that Theorem \ref{theo:WNLGO}
  includes other NLS type models with non-local
  perturbations than DS,  
provided the corresponding kernel $\widehat K$ is homogeneous of
degree zero and continuous away from the origin. A particular example
is given by the  
\emph{Gross-Pitaevskii equation for dipolar quantum gases}, i.e.
\begin{equation}
  \label{eq:dipole}
  i\d_t \psi + \frac{1}{2}\Delta \psi =  \mu \lvert
  \psi\rvert^2 \psi + \l\(K\ast \lvert\psi\rvert^2\)\psi,\quad x\in
  \R^3,\tag{DGP} 
\end{equation}
where the interaction kernel $K$ is given by 
\begin{equation}\label{dipolekernel}
  K(x) = \frac{1-3\cos^2\theta}{\lvert x\rvert^3}.
\end{equation}
Here $\theta=\theta(x)$ stands for the angle between $x\in \R^3$ and a
given dipole 
axis $n\in \R^3$, with $|n| =1$. In other words $\theta$ is defined via 
$\cos \theta= n\cdot x/\lvert x\rvert$.
In this case, we compute (see \cite{CMS08}), for $\xi\in\R^3\setminus\{0\}$,
\begin{equation*}
 \widehat K(\xi)= \frac{2}{3}(2\pi)^{5/2}
  \(3\cos^2\Theta-1\)
\end{equation*}
where $\Theta$ stands for the angle between $\xi$ and the dipole
axis. 
The model \eqref{eq:dipole} has been introduced in \cite{YiYou} in order to
describe (superfluid) Bose--Einstein condensates of particles with
large magnetic dipole moments. Note that for our analysis, we neglect
possible external potentials $V(x)$,  
usually present in physical experiments.
In this context, rescaling \eqref{eq:dipole} and studying the
asymptotics as $\eps \to 0$ correspond to the \emph{classical limit}
of quantum mechanics.  
\end{remark}

\subsection{Instability and norm inflation}\label{sec:inflation}

The insight gained in the proof of Theorem~\ref{theo:WNLGO} will allow
us to infer instability results of the Cauchy problem corresponding to
\eqref{eq:genNLS}. Let us remark, that  
the first rigorous result on the lack of well-posedness 
for the Cauchy problem of (NLS) in negative order Sobolev spaces goes
back to \cite{KPV01}, where the focusing 
case in $d=1$ was studied. This result was then generalized to $d\ge1$
in \cite{CCT2}, where  the lack of well-posedness for (NLS) from 
$H^s(\R^d)$ to $H^s(\R^d)$, has been proved for all $s<0$ (and
regardless of the sign of the nonlinearity). A general approach to
prove instability  
was  given in \cite{BeTa05}, where the authors studied the
\emph{quadratic} NLS.  
Applying their abstract result \cite[Proposition~1]{BeTa05} to the
models considered above, we prove a lack of well-posedness from
$H^s(\R^d)$ to $H^\si(\R^d)$ for \emph{all} $\si$.  

\begin{proposition}\label{prop:ill}
For all $s,\si<0$, the Cauchy problem for \eqref{eq:NLS}, with $d\ge2$, 
$\nu\in\N^\star$ and $\mu\neq0$  is ill-posed from $H^s(\R^d)$ to
$H^\si(\R^d)$. The same holds true for the Cauchy problem of 
\eqref{eq:DS0}, provided $\lambda+2\mu\neq0$, and for the one of 
\eqref{eq:dipole}, provided $(\lambda,\mu)\neq(0,0)$.
\end{proposition}

\begin{remark} Our result excludes the case $\l+2\mu=0$, which
  corresponds to the a situation in which the DS system is 
  known to be completely integrable, see e.g. \cite{AbCl}. The
  algebraic structure of the equation is indeed very peculiar then,
  since there exists a Lax pair.  
\end{remark}

The proof of Proposition \ref{prop:ill} is outlined in
Appendix~\ref{sec:ill}, following the ideas of \cite{BeTa05}.  Note,
however, that this approach is not 
constructive as it relies on the norm inflation for the first Picard
iterate. As we shall see in Section~\ref{sec:outline}, weakly
nonlinear geometric optics indeed allows us to infer a stronger result
than the one stated above, namely \emph{norm inflation}. To this end,
let us recall that in \cite{CDS10}, weakly nonlinear geometric optics
was used to prove instability results for NLS equations on
$\T^d$. There, we used initial data corresponding to two non-zero
amplitudes $\alpha_0,\alpha_1$ (one of which carried no
$\eps$-oscillation).  
By proving a transfer of energy (inspired by the ideas from 
\cite{CCT2} and \cite{CCTper}) from high to low frequencies 
(i.e. the zero frequency in fact) we were able to conclude instability. 
In the present work we shall start from three non-zero initial modes,
which, via nonlinear interactions, will be shown to always generate
the zero mode. This phenomenon is geometrically possible as soon as a
multidimensional framework is considered. From this fact we shall  
infer norm inflation. The price to pay in this approach via WNLGO is
an unnatural condition  on the initial Sobolev space:
\begin{theorem}\label{theo:genNLSinflation}
Consider either \eqref{eq:NLS} with $d\ge 2$, $\mu\neq0$ 
and $\nu\in \N^*$, or \eqref{eq:DS0} with $\lambda+2\mu\neq0$, 
or \eqref{eq:dipole} with $(\lambda,\mu)\neq(0,0)$.  
We can find a sequence of initial data $(\varphi_n)_{n\in \N}\in \Sch(\R^d)$, 
with
\begin{equation*}
  \|\varphi_n\|_{H^{-1/(2\nu)}(\R^d)}\Tend n {+\infty} 0,
\end{equation*}
and $t_n\to 0$ such that the solutions $\psi_n$ 
with $\psi_{n\mid t=0}=\varphi_n$ satisfy 
\begin{equation*}
  \|\psi_n(t_n)\|_{H^\si(\R^d)}\Tend n {+\infty} +\infty,\quad \forall
  \si\in \R.  
\end{equation*}
\end{theorem}
Unlike in the general approach of \cite{BeTa05}, in the proof of
Theorem~\ref{theo:genNLSinflation}, we construct explicitly the sequence
$(\varphi_n)_{n\in \N} $, as well as an approximation of
$(\psi_n)_{n\in \N} $ (which 
can be deduced from WNLGO). Note however, that we require $d\ge 2$,  
since our proof demands a multidimensional setting. 
Also, the reason why we restrict ourselves to \eqref{eq:NLS}, 
\eqref{eq:DS0} and \eqref{eq:dipole} in the above instability 
results is that we do not exhibit adequate initial data for 
\eqref{eq:genNLS} in its full generality.  
In \cite{CCT2}, norm inflation for (NLS) was shown from 
$H^s$ to $H^s$, under the assumption $s\le-d/2$. 
Theorem~\ref{theo:genNLSinflation} improves this previous
result in three different aspects:
\begin{enumerate}
\item We consider a more general NLS type model, including non-local
  perturbations (in particular the DS system). 
\item The range of $s$ is larger, since we assume $s\le -1/(2\nu)$, in a
  setting where we always have $1/(2\nu)<d/2$ (recall that $d\ge 2$ by
  assumption).  
\item The target space is larger: \emph{all} the Sobolev norms
  become unbounded at the same time. 
\end{enumerate}

\begin{remark} In the case $0<s<s_c$, the norm inflation
result proved in \cite{CCT2} was improved to a loss of regularity
result in \cite{AlCa09,Th08}, after \cite{Lebeau05} in the case of the wave
equation (roughly speaking, one proves norm inflation from $H^s$ to
$H^\si$, with 
$\si>\si_0$ for some  $\si_0<s$).
In the periodic setting $x\in \T^d$, 
instability results (from $H^s(\T^d)$ to $H^\si(\T^d)$ for all $\si\in
\R$) were proved in \cite{CCTper} in the case $d=1$, and then generalized to
the case $d\ge 1$ in \cite{CDS10} (however, the phenomenon proved there is
just instability, not norm inflation). Finally, 
we also like to mention 
the beautiful result by Molinet \cite{Mo09} in the case $x\in \T$ 
and $\nu=1$, and the recent result by Panthee \cite{Pa-p} which shows
that the flow map for BBM equation fails to be continuous at the
origin from $H^s(\R)$ to ${\mathcal D}'(\R)$ for all $s<0$. 
\end{remark}

Our last result concerns the (generalized) NLS equation only,
i.e. \eqref{eq:genNLS} with  $\lambda =0$, 
where we can prove norm inflation for a larger range of Sobolev indices.

\begin{theorem}\label{theo:inflation2}
Let $d\ge 2$, $\nu\in \N^*$, $\mu\in\R^*$, $\l=0$,
and $s<-1/(1+2\nu)$.  
We can find a sequence of initial data $\varphi_n\in \Sch(\R^d)$, 
with
\begin{equation*}
  \|\varphi_n\|_{H^s(\R^d)}\Tend n {+\infty} 0,
\end{equation*}
and $t_n\to 0$ such that the solutions $\psi_n$ to \eqref{eq:genNLS}
with $\psi_{n\mid t=0}=\varphi_n$ satisfy 
\begin{equation*}
  \|\psi_n(t_n)\|_{H^\si(\R^d)}\Tend n {+\infty} +\infty,\quad \forall
  \si\in \R. 
\end{equation*}
\end{theorem}
\begin{remark} We believe that the restriction
$s<-1/(1+2\nu)$ is only due to our approach, and we expect the result to
hold under the mere assumption $s<0$. Note that for $\lambda=0$ and
nonuniform signs of the $\eta_j$'s  
  in $\Delta_\eta$, the above result concerns the hyperbolic NLS.
\end{remark}

To conclude this paragraph, we point out that negative order Sobolev
spaces may go against the intuition. In Section~\ref{sec:moreweakly}
we shall study an asymptotic regime (for $\eps \to 0$) where the 
nonlinearity is ``naturally'' negligible at leading order in, say, $L^2\cap
L^\infty$, but fails to be negligible in some negative order Sobolev
spaces. This strange behavior of 
negative order Sobolev spaces is further illustrated by very basic
examples given in Appendix~\ref{sec:sobolev}.  
\medskip

\textbf{Notation.}
Let $(\Lambda^\eps)_{0<\eps\le 1}$ and $(\Upsilon^\eps)_{0<\eps\le 1}$
be two families of positive real numbers. 
\begin{itemize}
\item We write $\Lambda^\eps \ll \Upsilon^\eps$ if
$\displaystyle \limsup_{\eps\to 0}\Lambda^\eps/\Upsilon^\eps =0$.
\item We write $\Lambda^\eps \lesssim \Upsilon^\eps$ if 
$\displaystyle \limsup_{\eps\to 0}\Lambda^\eps/\Upsilon^\eps <\infty$.
\item We write $\Lambda^\eps \approx \Upsilon^\eps$ (same order of
  magnitude) if $\Lambda^\eps \lesssim
  \Upsilon^\eps$ and $\Upsilon^\eps \lesssim \Lambda^\eps$. 
\end{itemize}

\section{Interaction of high frequency waves in NLS type models}

In this section we shall study the interaction of high frequency waves
for the generalized NLS type equation \eqref{eq:genNLS}. To this end,
we  
shall first recall some results from  \cite{CDS10}, where the usual
case of NLS with elliptic dispersion is treated.  

\subsection{Geometric optics for elliptic NLS} 
\label{sec:NLSWNLGO} We consider the equation
\begin{equation}
  \label{eq:NLSsemi}
  i\eps \d_t u^\eps +\frac{\eps^2}{2}\Delta u^\eps = \eps \mu 
  |u^\eps|^{2\nu} u^\eps,\quad x\in \R^d,
\end{equation} 
with $d\ge1$, $\mu\in \R$ and $\nu\in \N^\ast$. The initial data is
supposed to be given by a superposition of highly oscillatory plane
waves, i.e. 
\begin{equation}\label{eq:DSWNLGOini}
  u^\eps(0,x)= \sum_{j\in J_0}\alpha_j(x) e^{i\k_j\cdot x/\eps},
\end{equation}
where for some index set $J_0\subseteq \N$ we are given initial wave
vectors 
\begin{equation*}
  \Phi_0=\{\k_j \ ,\ j\in J_0\} 
\end{equation*}
and smooth, rapidly decaying
amplitudes $\alpha=(\alpha_j) \in \Sch(\R^d)$. We seek an approximation of the
exact solution $u^\eps$ in the following form 
\begin{equation}\label{eq:approx}
  u^\eps(t,x)\Eq \eps 0 u^\eps_{\rm app}(t,x)= \sum_{j\in
    J}a_{j}(t,x)e^{i \phi_j(t,x)/\eps}. 
\end{equation}
As we shall see in the next subsection, the characteristic phases
$\phi_j$ will be completely determined by the set of of \emph{relevant
  wave vectors}: 
\begin{equation*}
  \Phi=\{\k_j \ ,\ j \in J\} \supseteq \Phi_0,
\end{equation*}
In order to rigorously prove an approximation of the form
\eqref{eq:approx}, there are essentially four steps needed: 
\begin{enumerate}
\item Derivation of the set $\Phi$.
\item Derivation of the amplitude equations, determining the $a_j$'s.
\item Construction of the approximate solution.
\item Justification of the approximation. 
\end{enumerate}
 In this section, we address the first two steps only. The last two points
 are dealt with in Sections~\ref{sec:construct} and 
 \ref{sec:WNLGO}, respectively.  

 \subsubsection{Characteristic phases}
 \label{sec:NLSphases}
Plugging the approximation \eqref{eq:approx} into \eqref{eq:NLSsemi},
and comparing equal powers of $\eps$, we find that the leading order
term is of order 
$\O(\eps^0)$. It can be made identically zero, if for all $j\in J$:
\begin{equation}\label{eq:NLSeikonal}
  \d_t\phi_j + \frac{1}{2}|\nabla \phi_j|^2 =0 .
\end{equation}
This \emph{eikonal equation} determines the characteristic phases
$\phi_j(t,x)\in \R$, entering the approximation \eqref{eq:approx}. In
view of \eqref{eq:DSWNLGOini},  
the eikonal equation is supplemented with initial data
$\phi_j(0,x)=\k_j\cdot x$, from which we can compute explicitly  
\begin{equation} \label{eq:NLSphi_j}
  \phi_j(t,x)= \k_j\cdot x-\frac{t}{2}|\k_j|^2. 
\end{equation}
Next, let $(\k_{\ell_1},\dots,\k_{\ell_{2\nu+1}})$ be a set of given
wave vectors. The  
corresponding nonlinear interaction in $|u^\eps|^{2\nu}u^\eps$ is then given by
\begin{equation*}
  a_{\ell_1}\overline a_{\ell_2} \dots
  a_{\ell_{2\nu+1}}e^{i(\phi_{\ell_1}- \phi_{\ell_2}+\dots
    +\phi_{\ell_{2\nu+1}})/\eps}. 
\end{equation*}
The resulting phase $\phi = \phi_{\ell_1}- \phi_{\ell_2}+\dots
    +\phi_{\ell_{2\nu+1}}$ satisfies the eikonal 
equation \eqref{eq:NLSeikonal}, and thus needs to be taken into
account in our approximation, if there exists $\k\in \R^d$ such that: 
\begin{equation} \label{eq:NLSreson}
  \sum_{k=1}^{2\nu+1}(-1)^{k+1}\k_{\ell_k}=\k\text{ and
  }\sum_{k=1}^{2\nu+1}(-1)^{k+1}|\k_{\ell_k}|^2=|\k|^2.  
\end{equation}
With $\k_j=\k$, a phase $\phi_j$ of the form \eqref{eq:NLSphi_j} is
then said to be generated through a \emph{resonant interaction}
between the  phases $(\phi_{\ell_k})_{1\le k\le 2\nu+1}$. 

This yields the following algorithm to construct 
the set $\Phi$ from $\Phi_0$: Starting from the initial (at most 
countable) set $\Phi_0=\{\k_j \ ,\ j \in J_0\}$, we obtain a first 
generation $\Phi_1=\{\k_j \ ,\ j \in J_1\} \supset \Phi_0$ 
(with $J_1 \supset J_0$) by adding to $\Phi_0$ 
all points $\k\in\R^d$ satisfying \eqref{eq:NLSreson} for some 
$\{ \phi_{\ell_1},\dots,\phi_{\ell_{2\nu+1}} \} \subset \Phi_0$. 
By a recursive scheme, we are led to a set 
$\Phi\ = \{ \k_j \ ,\ j \in J \}$ which is (at most countable and) stable 
under the resonance condition \eqref{eq:NLSreson}. 

\begin{remark}
It is worth noting that $\Phi$ is a subset of the group generated 
by $\Phi_0$ (in $(\R^d,+)$). In particular, if $\Phi_0 \subset \Z^d$, 
then $\Phi \subset \Z^d$. 
\end{remark}

It turns out that we do not need a precise description of the 
resonances to prove norm inflation stated in
Theorem~\ref{theo:genNLSinflation} or 
Theorem~\ref{theo:inflation2}. However, 
in the case of a cubic nonlinearity $\nu =1$, all possible resonances
can be easily described geometrically by the following lemma
(\cite{Iturbulent,CDS10}).  
To this end, we denote for $j\in J$, the set of all  resonances by
\begin{equation*}
  I_j=\Big \{ \(\ell_1,\ldots,\ell_{2\nu+1}\)\in J^{2\nu+1} \mid
    \sum_{k=1}^{2\nu+1}(-1)^{k+1} \k_{\ell_k} = \k_j,\  \sum_{k=1}^{2\nu+1}(-1)^{k+1}
    |\k_{\ell_k}|^2 = |\k_j|^2\Big\}. 
\end{equation*}

\begin{lemma} \label{lem:rectangles}
Let $\nu=1$, $d\ge2$, and $j,k,\ell,m$ belong to $J$. Then,
$(\k_k,\k_\ell,\k_m) \in I_j$  
precisely when the endpoints of the vectors $\kappa_k, \kappa_\ell, \kappa_m,
\kappa_j$ form four corners of a non-degenerate rectangle
with $\kappa_\ell$ and $\kappa_j$ opposing each other, or when this
quadruplet corresponds to one of the two following degenerate cases:
$(\k_k=\k_j, \k_m=\k_\ell)$,  
or $(\k_k=\k_\ell, \k_m=\k_j)$. 
\end{lemma}

\begin{example}\label{ex:key}
The proof of norm inflation will be based upon the following case. Let 
\begin{equation}\label{eq:J0}
  \Phi_0=\left\{ \k_1=(1,0,\dots,0), \ \k_2=(1,1,0,\dots,0),\
  \k_3=(0,1,0,\dots,0) \right\} \subset \R^d.
\end{equation}
The above lemma shows that for the cubic nonlinearity cubic ($\nu=1$), 
the set of relevant phases is simply
\begin{equation*}
  \Phi=\Phi_0\cup \{\k_0=0_{\R^d}\}.   
\end{equation*}
One and only one phase is created by resonant interaction: the zero
phase. For higher order nonlinearities ($\nu>1$), we also 
have $0\in\Phi$ (since $0 = -\k_1+\k_2-\k_3+(\k_1-\k_1+\dots-\k_1)$).
\end{example}

 \subsubsection{The amplitudes system}
 \label{sec:NLSampli}
Continuing the formal WKB approach, the $\O(\eps^1)$ term yields, after
projection onto characteristic oscillations $e^{i\phi_j/\eps}$, a
system of transport equations:
\begin{equation}
  \label{eq:transportNLS}
 \d_t a_j+\k_j\cdot \nabla a_j = -i\mu
 \sum_{(\ell_1,\dots,\ell_{2\nu+1})\in
   I_j}a_{\ell_1}\overline a_{\ell_2}\dots a_{\ell_{2\nu+1}}\quad , 
 \quad a_{j\mid t=0}=\alpha_j, 
\end{equation}
with the convention $\alpha_j=0$ if $j\not\in J_0$. 
As claimed above, it is not necessary to fully understand the resonant
sets to prove Theorem~\ref{theo:genNLSinflation} or
Theorem~\ref{theo:inflation2}.  The following lemma will suffice:
\begin{lemma}\label{lem:a0NLS}
  Let $\nu\in \N^*$, $\mu\in\R^*$ and $d\ge 2$. Assume $\Phi_0$ 
  is given by \eqref{eq:J0}. There exist $\alpha_1,\alpha_2,\alpha_3
  \in \Sch(\R^d)$ such that if we set $\k_0=0_{\R^d}$, 
  \eqref{eq:transportNLS} implies 
  \begin{equation*}
    \d_t a_{0\mid t=0} \not =0. 
  \end{equation*}
For instance, this is so if $\alpha_1=\alpha_2=\alpha_3\not =0$. 
\end{lemma}
\begin{proof}
  Assume
  $\alpha_1=\alpha_2=\alpha_3=\alpha$. Equation~\eqref{eq:transportNLS}
  yields
  \begin{align*}
   \d_t a_{0\mid t=0} 
   & = -i\mu  \sum_{(\ell_1,\dots,\ell_{2\nu+1})\in I_0}
   \alpha_{\ell_1}\overline \alpha_{\ell_2}
   \dots\alpha_{\ell_{2\nu+1}}
= -i\mu C(\nu,d)|\alpha|^{2\nu}\alpha.
  \end{align*}
Then, $C(\nu,d)\ge 1$, since $\(1,2,3,1,1,\dots,1\)\in I_0$. 
\end{proof}
This lemma shows that even though the zero mode is absent at time
$t=0$, it appears instantaneously for a suitable choice of the initial
amplitudes $\alpha_1,\alpha_2,\alpha_3$. This is one
of the keys in the proof of the results presented in
Section~\ref{sec:inflation}.  
\begin{remark}
  This result fails to be true in the one-dimensional cubic case
  $d=\nu=1$, and in a situation where one starts with 
  only two (non-trivial) modes ($d,\nu\ge 1$, but $\sharp \{j\in J_0\
  \mid \alpha_j\not =0\}\le 
  2$). In both cases no new resonant mode can be created through the
  nonlinear interaction (see \cite{CDS10}).  
\end{remark}

\subsection{Geometric optics for the DS system} \label{sec:DSWNLGO} In
order to apply our approach to the DS system \eqref{eq:DS}, we first
have to  
understand the resonance structure for $\eta = - 1$ (a case where
$\Delta_\eta$ is an hyperbolic operator). We shall, as a first step, neglect the
action of the non-local term $E$  
and instead consider
\begin{equation}\label{eq:etaNLS}  
  i\eps \d_t u^\eps + \frac{\eps^2}{2}\(\eta  \d^2_{x_1} + \d^2_{x_2}
  \) u^\eps= \eps \mu |u^\eps|^{2} u^\eps\ ,  
  \quad x\in \R^2,
\end{equation}
subject to oscillatory initial data of the form
\eqref{eq:DSWNLGOini}. 

\subsubsection{Characteristic phases and resonances} 
\label{sec:phasesDS} 
We follow the same steps as in Section~\ref{sec:NLSphases} and
determine the characteristic phases via 
\begin{equation}
  \label{eq:hj}
   \d_t \phi+\frac{1}{2} \Big( \eta (\partial_{x_1} \phi)^2 +
   (\partial_{x_2} \phi)^2 \Big)=   0 \quad , 
   \quad \phi(0,x) =  \k\cdot x. 
\end{equation}
Denoting $\k= (p,q)$, the solution of this equation is given by
\begin{equation}\label{eq:planewave}
\phi_j(t,x) = px_1  + q x_2 - \frac{t}{2}\( q^2 + \eta p^2 \).
\end{equation}
In the case $\eta = -1$ we see that if initially $\k=(\pm p,\pm p) $, then 
$\phi(t,x) = \k \cdot x$ is independent of time.  

In order to understand possible resonances due to the cubic
nonlinearity, we simply notice  
that, if some phases $\phi_k$, $\phi_\ell$ and $\phi_m$ are given by 
\eqref{eq:planewave} (with $\k$ equal to $\k_k$, $\k_\ell$ and 
$\k_m$, respectively), then the combination 
$\phi = \phi_k - \phi_\ell +\phi_m$ again solves \eqref{eq:hj} if, and
only if,
it is of the form given by \eqref{eq:planewave}, with 
$\k=(p,q)\in\R^2$ satisfying 
\begin{equation}\label{eq:resonance}
\k = \k_k-\k_\ell +\k_m \quad ,\quad  
q^2 + \eta p^2= q_k^2 - q_\ell^2 +
q_m^2 + \eta (p_k^2 - p_\ell^2 + p_m^2). 
\end{equation}
Thus, the same iterative procedure as in Section~\ref{sec:NLSphases} 
allows us to build from a given (at most countable) set of wave-vectors 
$\Phi_0=\{\k_j \ ,\ j \in J_0\}$ in $ \R^2$, a new set 
$\Phi=\{\k_j \ ,\ j \in J\}$, closed under the resonance 
condition \eqref{eq:resonance}. Again, it is worth noting that $\Phi$ 
is a subset of the group generated by $\Phi_0$ (in $(\R^2,+)$). 
In particular, if $\Phi_0 \subset \Z^2$, then $\Phi \subset \Z^2$. 
We consequently alter the definition of the resonance set $I_j$ given
above and denote,  
for all $j \in J$:
\begin{equation*}
 I _j =\{  (k,\ell,m) \in J^3\ \mid 
 \k_j = \k_k-\k_\ell +\k_m , \,  
q_j^2 + \eta p_j^2= q_k^2 - q_\ell^2 +
q_m^2 + \eta (p_k^2 - p_\ell^2 + p_m^2)  \}. 
\end{equation*}

The case $\eta=1$ has already been discussed in Section~\ref{sec:NLSphases}. 
To understand better the non-elliptic case $\eta = -1$ we first note
that \eqref{eq:resonance} is equivalently fulfilled by $(\k_k - \k,
\k_\ell - \k, \k_m - \k)$, as can easily be checked by a
direct computation. 
Thus it is enough to understand the case where the zero mode 
$\k =(0,0)$ is created. In this case the resonance condition
\eqref{eq:resonance} is equivalent to  
\begin{equation*}
 \k_\ell = \k_k + \k_m , \ \text{with $\k_k, \k_m$ satisfying: } \ 
q_kq_m = p_kp_m. 
\end{equation*}
This leads to the following statement:

\begin{lemma}
Let $\eta=-1$, and $j,k,\ell,m$ belong to $J$. 
Then, $(\k_k,\k_\ell,\k_m) \in I_j$  
precisely when $(\k_k=\k_j, \k_m=\k_\ell)$,  
or $(\k_k=\k_\ell, \k_m=\k_j)$, or when the endpoints of the 
vectors $\k_j$, $\k_k$, $\k_\ell$, $\k_m$ form four corners of a 
non-degenerate parallelogram, with 
$\kappa_\ell$ and $\kappa_j$ opposing each other, and such that  
$(\kappa_k-\kappa_\ell)/|\kappa_k-\kappa_\ell|$ and 
$(\kappa_m-\kappa_\ell)/|\kappa_m-\kappa_\ell|$ are symmetric 
with respect to the first bisector.
\end{lemma}

The resonance condition is different from the elliptic case $\eta =
+1$ (Lemma~\ref{lem:rectangles},
based on the completion of rectangles). In particular resonances for
$\eta =+1$ are not necessarily also resonances for $\eta = -1$ and
vice versa, as illustrated by the examples below.  
The common feature of the two cases is that it takes at
least three different phases to create  a new one 
via the cubic nonlinearity. 

\begin{example} 
  Let $\k_k=(2,1)$, $\k_m=(1,2)$ and
  $\k_\ell=(3,3)$. 
  Then the origin is obtained by cubic
  resonance in the hyperbolic case $\eta=-1$, while no new phase
  results of the interaction of these three phases in the elliptic
  case.  
\end{example}

\begin{example} On the other hand, if $\k_k=(0,0)$, $\k_\ell=(1,1)$ and
  $\k_m=(0,2)$, we obtain $\k_k-\k_\ell+\k_m = (-1,1)$: in the elliptic case
  $\eta=+1$, this is resonance, 
while it is not in the case $\eta = -1$.
\end{example}

\begin{example}
  With the approach we have in mind to prove
Theorem~\ref{theo:genNLSinflation}, let $\Phi_0$ be given by
\eqref{eq:J0}. 
Like in the elliptic case, $\Phi=\Phi_0\cup\{0_{\R^2}\}$ is obtained
by cubic resonance when $\eta=-1$.

\end{example}\subsubsection{Oscillatory structure of the non-local term} 
\label{sec:nonloc}

In order to proceed further, we need to take into account the Fourier
multiplier $E$ defined in \eqref{eq:E}. The corresponding nonlinearity is given by 
\begin{equation*}
  F\(u^\eps_1,u^\eps_2,u^\eps_3\) = E\(u^\eps_1\overline {u}^\eps_2\)u^\eps_3.
\end{equation*} 
Having in mind \eqref{eq:approx} the point is to understand the rapid
oscillations of  
\begin{equation*}
  F\(a_k e^{i\phi_k /\eps}, a_\ell e^{i\phi_\ell /\eps},a_m e^{i\phi_m/\eps}\),
\end{equation*}
with $\phi_k$, $\phi_\ell$ and $\phi_m$ satisfying 
\eqref{eq:planewave}. The phase
$\phi_m$ obviously factors out and so do the oscillations in time,  
since they are not affected by the action of $E$. In view of the
discussion on resonances, we must understand the high frequency
oscillations of 
\begin{equation*}
  E\(a_k \overline a_\ell e^{i\kappa\cdot x/\eps}\),\text{ where }\kappa =
  \kappa_k-\kappa_\ell. 
\end{equation*}
If $\kappa=0$, there is no rapid oscillation, and we can directly
resume the argument of the case $E={\rm Id}$. If $\kappa\not =0$, we
denote $b =a_k \overline a_\ell$ and  write
\begin{align*}
  e^{-i\kappa\cdot x/\eps}E\(a_k \overline a_\ell e^{i\kappa\cdot
    x/\eps}\)&= \frac{1}{(2\pi)^2}\int_{\R^2}\int_{\R^2} e^{i
    (x-y)\cdot \xi} \widehat K(\xi) e^{i\kappa\cdot
      (y-x)/\eps}b(y)dyd\xi\\
& = \frac{1}{(2\pi\eps )^2}\int_{\R^2}\int_{\R^2} e^{i
    (x-y)\cdot (\zeta-\kappa)/\eps} \widehat K(\zeta)
  b(y)dyd\zeta,   
\end{align*}
where we have used the fact that the function $\widehat K$ is
$0$-homogeneous. Denote by
$I^\eps(x)$ the above integral.
Applying formally the stationary phase formula yields:
\begin{equation*}
  I^\eps(x) \Eq \eps 0 \widehat K(\k)b(x) = \widehat K(\k_k-\k_\ell)a_k(x)
  \overline a_\ell(x). 
\end{equation*}
This formal argument suggests that the non-local operator indeed acts
like a cubic nonlinearity when $\eps \to 0$, and reveals a formula for
the corresponding amplitude system.  
A rigorous proof for this argument will be given later in
Section~\ref{sec:WNLGO}. For the moment, we shall proceed formally by
plugging the ansatz \eqref{eq:approx} into  
\begin{equation}\label{eq:DSsemi}  
  i\eps \d_t u^\eps + \frac{\eps^2}{2}\(\eta  \d^2_{x_1} + \d^2_{x_2}
  \) u^\eps= \eps \lambda E(|u^\eps|^2) u^\eps\ + 
  \eps \mu |u^\eps|^2 u^\eps\ ,  
  \quad x\in \R^2.
\end{equation}
The terms of order $\O(\eps^0)$ are identically zero since all the
$\phi_j$'s are characteristic.  
For the $\O(\eps^1)$ term, we project onto characteristic
oscillations, to obtain the following system of (non-local) transport
equations 
\begin{equation}\label{eq:transportsystem}
  \begin{aligned}
    \d_t a_j +\( \eta p_j \partial_{x_1} + q_j \partial_{x_2} \) a_j 
  & =  -i\lambda \sum_{\ell \in J} E(|a_\ell|^2) a_j \\
  & \quad -i \lambda \sum_{{(k,\ell,m)\in I_j}\atop{ \ell\not =k}}
    \widehat K(\k_k-\k_\ell) a_k \overline   a_\ell a_m \\
  & \quad -i \mu \sum_{{(k,\ell,m)\in I_j}} a_k \overline a_\ell a_m,  
  \end{aligned}
  \end{equation}
subject to initial data $ a_j(0,x) =    \alpha_j ( x)$. 
\begin{remark}\label{rem:SP}
  The above computation suggests that if $\widehat K$ is
  $M$-homogeneous with $M<0$, then the second line in
  \eqref{eq:transportsystem} vanishes in the limit $\eps>0$, since
  $\eps^{-M}$ can be factored out. This
  argument is made rigorous in \S\ref{sec:localizing}. If in addition
  there is no local nonlinearity, i.e. $\mu=0$, \eqref{eq:transportsystem}
  takes the form
  \begin{equation*}
    \d_t a_j +\k_j\cdot \nabla_\eta a_j 
  =  -i\lambda \sum_{\ell \in J} E(|a_\ell|^2) a_j. 
  \end{equation*}
Then the modulus of $a_j$ is constant along the characteristic curves,
along which the above equation is of the form $D_t a_j = 
a_j\times i\R$, so $D_t|a_j|^2=0$. In particular, no mode is
created in this case, a situation to be compared with the framework of
\cite{GMS-p}, where the assumptions made on $K$ are of a different
kind. An important example where this remark applies is the
Schr\"odinger--Poisson system ($d\ge 3$), for which $M=-2$. We
therefore strongly believe that also in this case  
one can prove a result analogous to Theorem~\ref{theo:WNLGO} (with an
error rate $\O(\eps)$, since the second line in 
\eqref{eq:transportsystem} becomes 
$\O(\eps^{-M})=\O(\eps^2)$, and we will see in Section~\ref{sec:nonstat} that
non-resonant phases generate an error of order $\O(\eps)$). However,
we expect that the functional setting has to be slightly modified,
since  
the Wiener algebra framework (used to prove Theorem~\ref{theo:WNLGO})
may no longer be convenient due to $\widehat K\not \in L^\infty$.  
Nevertheless, a setting based on $L^2$ spaces should be sufficient 
(to derive approximation in $L^2$, but not in $L^\infty$), 
since the Poisson nonlinearity is $L^2$-subcritical. 
\end{remark}
Similarly to
Lemma~\ref{lem:a0NLS}, with
$\Phi_0$ as in \eqref{eq:J0}, and we can find corresponding initial data
$\alpha_1,\alpha_2,\alpha_3$ in $\Sch(\R^2)$,  
such that the zero mode appears instantaneously, \emph{provided} that
$\lambda+2\mu\not =0$.
\begin{lemma}\label{lem:a0DS}
  Let $\eta=\pm1$. Assume $\Phi_0$ is given by \eqref{eq:J0}, 
  and set $\k_0=0_{\R^2}$. The following are equivalent: 
  \begin{itemize}
  \item[(i)] $\lambda + 2 \mu \neq 0$. 
  \item[(ii)] There exist $\alpha_1,\alpha_2,\alpha_3\in
  \Sch(\R^2)$ such that \eqref{eq:transportsystem} implies 
  $\d_t a_{0\mid t=0} \not =0$. 
  \end{itemize}
When $\lambda + 2 \mu \neq 0$, $\alpha_1,\alpha_2,\alpha_3$ are 
admissible if and only if there is $x\in\R^2$ such that
$\alpha_1(x)\alpha_2(x)\alpha_3(x) \not =0$.  
\end{lemma} 

\begin{proof}
For this choice of $\Phi_0$, setting $\k_0=0$, we have $I_0\neq
\emptyset$. Since the only $(k,\ell,m)\in I_0$ corresponding 
to possibly non-zero products 
$\alpha_k\overline\alpha_\ell\alpha_m \not =0$ are $(1,2,3)$ and 
$(3,2,1)$, we have from \eqref{eq:transportsystem}:
 \begin{align*}
    \d_t {a_0}_{|_{t=0}}&=  - i \left( \lambda 
 \left( \widehat K(\k_1-\k_2) + \widehat K(\k_3-\k_2) \right)
 + 2 \mu \right) \alpha_1\overline\alpha_2\alpha_3\\
&=  - i \left( \lambda 
 \left( \widehat K(\k_3) + \widehat K(\k_1) \right)
 + 2 \mu \right) \alpha_1\overline\alpha_2\alpha_3.
  \end{align*}
Now, since $p_1p_3+\eta q_1q_3=0$, we one easily checks the identity
$\widehat K(\k_3) 
+ \widehat K(\k_1) =1$ and thus
\begin{align*}
\d_t {a_0}_{|_{t=0}} = - i \, (\lambda+2\mu) \, \alpha_1\overline\alpha_2\alpha_3.
\end{align*}
The lemma follows.
\end{proof}
\begin{remark}
  In the case $\l+2\mu=0$ (integrable case), one can prove by
  induction that $\d_t^m  a_{0\mid t=0}=0$ for all $m\in \N$. Thus, the zero
  mode does not appear, at least if we consider a smooth (analytic)
  setting. Note, that this aspect is not due to our initial choice $\Phi_0$. If
  $\l+2\mu=0$, the zero mode cannot appear, regardless of the precise
  form of $\Phi_0$. Indeed, grouping the sets of three phases creating the
  zero mode (with of course $\alpha_0=0$), we may assume that we
  consider only three initial phases: 
  the point is to notice that if
  $\k_1-\k_2+\k_3=0$, we still have
  \begin{equation*}
    \d_t {a_0}_{|_{t=0}}=  - i \left( \lambda 
 \left( \widehat K(\k_3) + \widehat K(\k_1) \right)
 + 2 \mu \right) \alpha_1\overline\alpha_2\alpha_3,
  \end{equation*}
and we conclude as in the proof of Lemma~\ref{lem:a0DS}.
\end{remark}

\subsection{Possible generalizations}
\label{sec:extension}

As far as geometric optics is concerned (i.e. determining the
resonances and deriving the corresponding amplitude system),  
the above analysis can be reproduced without modification (except
notations) in the case of: 
\begin{equation}\label{eq:nlsgen}
  i\eps\d_t u^\eps +\frac{\eps^2}{2}\Delta_\eta u^\eps =  \eps \l 
  \(K\ast|u^\eps|^{2\nu}\)u^\eps + \mu \eps  |u^\eps|^{2\nu} u^\eps ,
  \quad x\in \R^d, \, \nu \in \N,
\end{equation}
provided $\widehat K(\xi)$ is homogeneous of degree zero and continuous
away from the origin. In this case, the characteristic phases are  
given by 
\begin{equation*}
\phi_j (t,x) = \k_j \cdot x 
- \frac{t}{2} \sum_{\ell=1}^d \eta_\ell \k_{j, \ell}^2 ,
\end{equation*}
and the corresponding system of transport equations reads:
\begin{equation}\label{eq:transportsystemgen}
  \begin{aligned}
    \d_t a_j + \sum_{\ell=1}^d 
    \eta_{\ell}\k_{j,\ell}\cdot \partial_{x_\ell} a_j 
  & =  -i\lambda \sum_{(\ell_1,\dots,\ell_{2\nu}) \in J^{2\nu} , 
   \atop{ (\ell_1,\dots,\ell_{2\nu},j) \in I_j}}
   E(a_{\ell_1} \overline a_{\ell_2} \dots \overline a_{\ell_{2\nu}}) 
   a_j \\ 
  &  -i \lambda 
   \sum_{{(\ell_1,\dots,\ell_{2\nu+1})\in I_j} ,
   \atop{ \ell_{2\nu+1}\not =j}} 
   \widehat K(\k_j-\k_{\ell_{2\nu+1}}) 
    a_{\ell_1} \overline a_{\ell_2} \dots a_{\ell_{2\nu+1}} \\
  & -i \mu \sum_{(\ell_1,\dots,\ell_{2\nu+1})\in I_j} 
  a_{\ell_1} \overline a_{\ell_2} \dots a_{\ell_{2\nu+1}}  ,
  \end{aligned}
  \end{equation}
where we denote $E(f) = K\ast f$ and
\begin{equation*}
I _j = \{ (\ell_1,\dots,\ell_{2\nu+1}) \in J^{2\nu+1} \mid 
\phi_j = \phi_{\ell_1}-\phi_{\ell_2}+\dots+\phi_{\ell_{2\nu+1}} \}. 
\end{equation*}
The only point one has to check so as to derive instability results 
from the geometric optics result is that there exist initial data 
$\alpha_j$, such that $\d_t{a_0}\not =0$. When $\lambda=0$, the 
same proof as for Lemma~\ref{lem:a0NLS} yields
\begin{lemma}\label{lem:a0genNLS}
  Let $\nu\in \N^*$, $\lambda=0$, $\mu\in\R^*$ and $d\ge 2$. 
  Assume $\Phi_0$ is given by \eqref{eq:J0}. There exist 
  $\alpha_1,\alpha_2,\alpha_3 \in \Sch(\R^d)$ such that if we set 
  $\k_0=0_{\R^d}$, \eqref{eq:transportsystemgen} implies 
  \begin{equation*}
    \d_t a_{0\mid t=0} \not =0. 
  \end{equation*}
For instance, this is so if $\alpha_1=\alpha_2=\alpha_3\not =0$. 
\end{lemma}
We can also take the non-local term into account, at least for the 
cubic nonlinearity and standard ``elliptic'' Schr\"odinger operator. 
In the case of dimension $d=2$, for all $c\in\R$, introduce 
\begin{equation*}
\mathcal{E}_c = 
\big\{ L \in C(\Sph^1,\C) \mid \,
\forall \xi\in \Sph^1, \, L\(\xi e^{i\pi/2}\) = c - L(\xi) \big\}.
\end{equation*}
There is a one-to-one correspondence, through $L \mapsto M$, 
where 
$$\forall \, \theta\in\R, \quad M(\theta)=L\(e^{i\theta}\)-c/2,$$ 
between $\mathcal{E}_c$ and the space of continuous functions 
on $\R$, having the ``anti'' ($\pi/2$)-periodicity symmetry 
$M(\theta+\pi/2)=-M(\theta)$. Lemma~\ref{lem:a0DS} is a particular 
case of: 
\begin{lemma} \label{lem:a0nonlocNLS} 
Consider $K\in\Sch'(\R^d)$ such that $\widehat K \in L^\infty(\R^d)$ 
is homogeneous of degree $0$, and  continuous away from the origin,  
but not constant on $\R^d\setminus\{0\}$. 
Let $\nu=1$, and $\lambda,\mu\in\R$, with $\lambda\neq0$. 
Set $\kappa_0 = 0_{\R^d}$.  
\begin{enumerate}
\item In the case $d\ge3$, there exist 
$\kappa_1,\kappa_2,\kappa_3 \in\R^d\setminus\{0\}$ 
and $\alpha_1,\alpha_2,\alpha_3\in\Sch(\R^d)$ 
such that, with $\Phi_0 = \{ \kappa_1,\kappa_2,\kappa_3 \}$, 
\eqref{eq:transportsystemgen} implies $\d_t a_{0\mid t=0} \not =0$.  
\item In the case $d=2$, the following are equivalent\emph{:} \\
\emph{(i)} $\widehat K \notin \mathcal{E}_{-2\mu/\lambda}$. \\
\emph{(ii)} There exist $\kappa_1,\kappa_2,\kappa_3 \in\R^2\setminus\{0\}$ 
and $\alpha_1,\alpha_2,\alpha_3\in\Sch(\R^2)$ 
such that, with $\Phi_0 = \{ \kappa_1,\kappa_2,\kappa_3 \}$, 
\eqref{eq:transportsystemgen} implies $\d_t a_{0\mid t=0} \not =0$.
\end{enumerate}  
Furthermore, in the cases where $\d_t a_{0\mid t=0} \not =0$ 
is possible, $\alpha_1,\alpha_2,\alpha_3$ are 
admissible if and only if there is $x\in\R^d$ such that
$\alpha_1(x)\alpha_2(x)\alpha_3(x) \not =0$.
\end{lemma}
\begin{proof}
Finding $\kappa_1,\kappa_2,\kappa_3 \in\R^d\setminus\{0\}$ 
generating $\k_0=0$ by resonance amounts to finding two non-zero 
and mutually orthogonal vectors, $\k_1$ and $\k_3$: $\k_2$ is then 
determined when forming the rectangle $(\k_0,\k_1,\k_2,\k_3)$. 
Once such vectors are chosen, since the only $(k,\ell,m)\in I_0$ 
corresponding to possibly non-zero products 
$\alpha_k\overline\alpha_\ell\alpha_m$ are $(1,2,3)$ 
and $(3,2,1)$, we have from \eqref{eq:transportsystem}:
\begin{equation*}
\d_t {a_0}_{|_{t=0}} =  - i \left( \lambda 
\left( \widehat K(\k_1-\k_2) + \widehat K(\k_3-\k_2) \right)
+ 2 \mu \right) \alpha_1\overline\alpha_2\alpha_3 . 
\end{equation*}
This yields $\d_t {a_0}_{|_{t=0}} \not = 0$ if and only if we are able to find 
two non-zero and mutually orthogonal vectors $\k$ ($=\k_1-\k_2$) 
and $\k'$ ($=\k_3-\k_2$) such that $\widehat K(\k) + \widehat K(\k') 
\not = -2\mu/\lambda$ (and in this case, the choice for 
$\alpha_1,\alpha_2,\alpha_3$ is clear). 

In dimension $d=2$, the possibility of finding such $\k$, $\k'$ is
equivalent to requiring $K \notin \mathcal{E}_{-2\mu/\lambda}$. 

In dimension $d\ge3$, suppose that for all choice of non-zero 
and mutually orthogonal $\k$ and $\k'$, the restriction of $\widehat 
K$ to the circle $\Sph^1$ centered at the origin, in the plane defined by 
$\{\k,\k'\}$, belongs to $\mathcal{E}_c$, with $c=-2\mu/\lambda$. 
Choosing a direction orthogonal to both $\k$ and $\k'$ defines the 
``north pole'' of an $\Sph^2$ sphere with the above circle as equator. 
Then, the value of $\widehat K$ at this pole must be $c - \widehat 
K(\k)$, as well as $c - \widehat K(\k')$, which implies 
$\widehat K(\k') = \widehat K(\k)$, and since 
$\widehat K(\k') = c - \widehat K(\k)$, we obtain that $\widehat K$ 
is constant (equal to $c/2$), and thus a contradiction.  
\end{proof}
This setting entails the case of the Gross-Pitaevskii equation for 
dipolar quantum gases \eqref{eq:dipole}, for which $d=3$. 
One can also easily generalize this result to higher-order 
nonlinearities ($\nu>1$), at least in dimension $d\ge3$. 
Combining non-elliptic Schr\"odinger operators with higher-order 
nonlinearities and non-local perturbations would lead to more tedious computations and we henceforth do not go into any further detail.

\section{Construction of the exact and approximate solutions}
\label{sec:construct}

\subsection{Analytical framework}
\label{sec:frame}

We first need to introduce  the  Wiener algebra framework 
similarly to what is used in \cite{CDS10}. 
\begin{definition}[Wiener algebra]
 We define
 \begin{equation*}
   W(\R^d)= \left\{ f\in {\mathcal S}'(\R^d,\C),\ \
     \|f\|_W:=\|\widehat f\|_{L^1(\R^d)} <+ \infty\right\}. 
 \end{equation*}
\end{definition}
The space  $W$ enjoys the following elementary properties (see 
\cite{CoLa09,CDS10}): 
\begin{enumerate}
\item $W$ is a Banach space, continuously embedded into
$L^\infty(\R^d)$.
\item $W$ is an algebra, and the
  mapping $(f,g)\mapsto fg$ is continuous from $W^2$ to
  $W$, with
  \begin{equation*}
    \|fg\|_W\le \|f\|_W\|g\|_W,\quad \forall f,g\in W.
  \end{equation*}
\item  For all $t\in \R$, the free Schr\"odinger group
  \begin{equation*}
    U^\eps(t)= \exp\left(i \eps \frac {t}{2} \, \Delta_\eta \right)  
\end{equation*}
is unitary on $W$.  
\end{enumerate}
In the following we shall seek an approximation result in $W\cap
L^2=\F(L^1)\cap \F(L^2)$. The basic idea is to first prove the result
in $W$ only  
and then  infer the
corresponding statement in $L^2$. We shall therefore use the extra
properties:
\begin{itemize}
\item[(2a)] By Plancherel formula and Young's inequality (or simply
  $\|f\|_{L^\infty}\le \|f\|_W$)
\begin{equation*}
    \forall f\in W, \ \forall g\in L^2(\R^d), \quad
    \|fg\|_{L^2(\R^d)}\le \|f\|_W\|g\|_{L^2(\R^d)}.
  \end{equation*}
\item[(3a)] For all $t\in \R$, the free Schr\"odinger group
 $   U^\eps(t)$
is unitary on $L^2(\R^d)$. 
\end{itemize}

\subsection{Existence results}
\label{sec:existence}
We first treat the case of the exact solution $u^\eps$, and then
address the construction of the approximate solution $u^\eps_{\rm
  app}$. 
\begin{lemma} 
Consider for $x\in \R^d$ the initial value problem 
\begin{equation}\label{eq:nlsivp}
  i\eps\d_t u^\eps +\frac{\eps^2}{2}\Delta_\eta u^\eps =  \eps \l 
  \(K\ast|u^\eps|^{2\nu}\)u^\eps + \mu \eps  |u^\eps|^{2\nu} u^\eps ,
  \quad u^\eps(0,x)=u_0^\eps(x), 
\end{equation}
where $\nu\in \N$, $\l, \mu \in \R$, 
and $K\in \Sch '(\R^d)$ is such that $\widehat K\in L^\infty(\R^d)$. 
If the initial data satisfies  $u_0^\eps\in W\cap L^2$, 
then there exist $T^\eps>0$ and a unique solution 
$u^\eps \in C([0,T^\eps];W\cap L^2)$ to \eqref{eq:nlsivp}. 
\end{lemma}

\begin{proof} The existence of a unique local in time solution in $W$ 
follows by combining the results of \cite[Proposition 5.8]{CDS10} 
and \cite[Lemma 3.3]{GMS-p}.  
In both cases property (2) together with the fact that  
$\widehat K\in L^\infty(\R^d)$  
implies that the nonlinearities are locally Lipschitz and the result
then follows by a standard fixed point argument.  
The existence of an $L^2$ solution then follows from the fact that 
for all $t\in [0,T^\eps]$, $ u^\eps(t,\cdot) \in L^\infty(\R^d)$, 
since $W \hookrightarrow L^\infty$.  
Therefore $|u^\eps|^{2\nu} \in L^\infty$ can be viewed as a bounded perturbation
potential, and using Plancherel formula we have 
\begin{align*}
\| \(K\ast |u|^{2\nu}\) v \|_{L^2} 
& \le \| K\ast |u|^{2\nu}\|_{L^\infty} \|v\|_{L^2} 
\le \| K\ast |u|^{2\nu}\|_W \|v\|_{L^2} \\
&\le \| \widehat K \|_{L^\infty}
\left\lVert \widehat{|u|^{2\nu}} \right\rVert _{L^1} \|v \|_{L^2} 
\le \| \widehat K \|_{L^\infty} 
\left\lVert  u \right\rVert_W^{2\nu} \| v\|_{L^2}. 
\end{align*}
Thus also the non-local term can be seen like a bounded perturbation
and the existence then follows by standard semi-group theory. 
\end{proof}

We now pass to an existence result for the transport system
\eqref{eq:transportsystem}. To this end we define the following
space for the amplitudes. 
\begin{definition}
Define  
$$
X(\R^d) = \{ a=(a_j)_{j\in J} \mid (\widehat{a}_j)_{j\in J}\in
\ell^1(J;L^1\cap L^2(\R^d)) \},$$
equipped with the norm
$$\|a\|_{X(\R^d)} = \sum_{j\in J}\big(\|\widehat{a}_j\|_{L^1} +
\|\widehat{a}_j\|_{L^2} \big).$$ 
For $s\in \N$, we define
\begin{equation*}
  X^s(\R^d)=\{ a \in X(\R^d)\mid \(\<\k_j\>^s a_j\)_{j\in J}\in
  X(\R^d)\text{ and } \d^\beta_x a \in X(\R^d),\quad
  \forall |\beta|\le s\},
\end{equation*}
equipped with the norm
$$\|a\|_{X^s(\R^d)} = \| \(\<\k_j\>^s a_j\)_{j\in J} \|_{X(\R^d)} 
+ \sum_{|\beta|\le s} \| \d^\beta_x a \|_{X(\R^d)}.$$ 
\end{definition}
We can state the following local in time existence result. 

\begin{lemma} \label{lem:existprof}
Let $d\ge 1$, $\nu\in \N\setminus\{0\}$, $\l, \mu \in \R$, 
and $K\in \Sch '(\R^d)$ such that $\widehat K\in L^\infty(\R^d)$. 
For all $\alpha=(\alpha_j)_{j\in J}\in X( \R^d)$, there exist $T>0$ 
and a unique solution 
\begin{equation*}
  t \mapsto a(t)=(a_j(t))_{j\in J}\in C([0,T],X( \R^d)),
\end{equation*}
to the transport system~\eqref{eq:transportsystemgen}, with
$a(0,x)=\alpha(x)$. Furthermore, the total mass is conserved: 
\begin{equation*}
  \frac{d}{dt}\sum_{j \in J}\|a_j(t) \|^2_{L^2} = 0.
\end{equation*} 
\end{lemma} 

\begin{proof} 
This result follows from the arguments given in  \cite{CDS10} (see
\cite[Lemma 3.1]{CDS10} for the last point). The main aspect to remark
is that the action of the nonlinear term $E$ raises no new difficulty,
in view of the estimate:
\begin{align*}
\| E(a_1 \bar a_2 \dots \bar a_{2\nu}) a_{2\nu+1} \|_X
& = \| \mathcal{F}(E (a_1 \bar a_2 \dots \bar a_{2\nu}))  \ast
\widehat {a}_{2\nu+1} \|_{L^1} \\ 
& \quad + \| \mathcal{F}(E (a_1 \bar a_2 \dots \bar a_{2\nu}))  \ast 
\widehat {a}_{2\nu+1} \|_{L^2}  \\
& \le  \| \mathcal{F}(E (a_1 \bar a_2 \dots \bar a_{2\nu})) \|_{L^1}  
\| a_{2\nu+1} \|_X  
\end{align*}
by Young's inequality. Since $\widehat {E (f g)} = (2\pi)^{-d/2}\widehat K
(\widehat f\ast \widehat g)$, we obtain
\begin{align*}
\| E(a_1 \bar a_2 \dots \bar a_{2\nu}) a_{2\nu+1} \|_X \le 
& \ \| \widehat K \|_{L^\infty}   
\| a_1 \|_X \dots \| a_{2\nu+1} \|_X.
\end{align*}
Since the same holds true for $E = {\rm Id}$ this shows that the
nonlinearity on the right hand side of \eqref{eq:transportsystem},
defines a continuous mapping from $X(\R^d)^{2\nu+1}$ to $X(\R^d)$.  
The existence of a local in-time solution 
then follows by a standard Cauchy-Lipschitz argument in the same way
as it is done in \cite{CDS10}. 
\end{proof}
At this stage, we have constructed the approximate solution
\begin{equation*}
  u_{\rm app}^\eps(t,x)=\sum_{j\in J}a_j(t,x)e^{i\phi_j(t,x)/\eps},
\end{equation*}
where the set $J$ and the corresponding $\phi_j$'s are like
constructed in Section~\ref{sec:NLSphases} and
Section~\ref{sec:phasesDS}, respectively, and 
the profiles are given by Lemma~\ref{lem:existprof}. Since
\begin{equation*}
  (a_j)_{j\in J} \in C([0,T],X(\R^d)),
\end{equation*}
we have in particular
\begin{equation*}
  u_{\rm app}^\eps \in C([0,T],W\cap L^2(\R^d)). 
\end{equation*}
More regularity will be needed in the justification of geometric optics.
\begin{lemma}\label{lem:morereg} 
Under the same assumption as in
  Lemma~\ref{lem:existprof} we have: 
\begin{enumerate}
  \item If $\alpha\in X^s(\R^d)$ for $s\in \N$, then the conclusions of
  Lemma~\ref{lem:existprof} remain true with $X(\R^d)$ replaced by
  $X^s(\R^d)$.
  \item Let  $\alpha\in X^2(\R^d)$. Then in addition 
\begin{equation*}
  t \mapsto a(t)=(a_j(t))_{j\in J}\in
C^1([0,T],X( \R^d)).
\end{equation*}
\end{enumerate}
\end{lemma}
\begin{proof} The
first point is straightforward. 
The second point stems from the first one, in view of the transport
equations, \eqref{eq:transportNLS}, \eqref{eq:transportsystem},
respectively.
\end{proof}
\begin{remark}\label{rem:forinflation}
  In particular, if the initial profiles $(\alpha_j)_{j\in
    J_0}$ belong to the Schwartz class, then $(a_j)_{j\in J}\in 
C([0,T],X^s( \R^d))$ for all $s\in \N$.  
\end{remark}

\section{Justification of multiphase geometric optics}
\label{sec:WNLGO}

In this section we justify the multiphase geometric optics
approximation. We assume 
\begin{equation*}
  u^\eps(0,x) = u^\eps_{\rm app}(0,x)=
\sum_{j\in J_0} \alpha_j(x) e^{i\k_j\cdot x/\eps},
\end{equation*}
with $(\alpha_j)_{j\in J_0}\in X(\R^d)$. The above 
Section~\ref{sec:construct} provides an approximate solution 
$u_{\rm app}^\eps \in C([0,T],W\cap L^2(\R^d))$.  With this existence
time $T$ (independent of $\eps$), we prove:
\begin{theorem}\label{theo:wnlgo}
Let $d\ge 1$, $\nu\in \N$, $\l, \mu \in \R$, and $E$ satisfy Assumption \ref{ass:E}. Let    
$\Phi_0\subset \Z^d$, with corresponding amplitudes $(\alpha_j)_{j\in J_0}\in X^2(\R^d)$. 

Then there exists $\eps_0>0$, such that for any $0<\eps\le \eps_0$, 
the solution to the Cauchy problem \eqref{eq:genWNLGO}-
\eqref{eq:genWNLGOini} satisfies
$u^\eps \in L^\infty([0,T];W\cap L^2)$. 
In addition, $u_{\rm app}^\eps$ approximates $u^\eps$ in the following sense
\begin{equation}\label{eq:error}
\left\lVert u^\eps-u_{\rm app}^\eps
\right\rVert_{L^\infty([0,T]; L^\infty\cap L^2)}\le  
\left\lVert u^\eps-u_{\rm app}^\eps
\right\rVert_{L^\infty([0,T];W\cap L^2)}\Tend \eps 0 0.
\end{equation}
When $\l=0$, $u_{\rm app}^\eps$ approximates $u^\eps$ up  
to $\O(\eps)$:
\begin{equation*}
\left\lVert u^\eps-u_{\rm app}^\eps
\right\rVert_{L^\infty([0,T]; L^\infty\cap L^2)}\le  
\left\lVert u^\eps-u_{\rm app}^\eps
\right\rVert_{L^\infty([0,T];W\cap L^2)}\lesssim \eps. 
\end{equation*}
\end{theorem}
\begin{remark}
  The above result can be proven (in the same way) in cases where the
  initial set of phases is not necessarily supported in $\Z^d$. We
  choose to prove 
  the approximation result in this peculiar framework since it is
  sufficient to infer Theorems~\ref{theo:genNLSinflation} and
  \ref{theo:inflation2}. A more general case would lead
 to small divisors problems, which can be treated as in \cite{CDS10}.
\end{remark}
Let $w^\eps = u^\eps - u_{\rm app}^\eps$ be the error term. 
From Section~\ref{sec:existence}, we know that there exists $T^\eps>0$ 
such that
\begin{equation*}
  w^\eps\in C([0,T^\eps];W\cap L^2).
\end{equation*}
We have to prove that for $\eps$ sufficiently small, $w^\eps\in
C([0,T];W\cap L^2)$, where $T>0$ stems from Lemma~\ref{lem:existprof},
together with \eqref{eq:error}. A standard continuity argument
shows that it suffices to prove \eqref{eq:error}. We compute
\begin{equation*}
  i\eps\d_t w^\eps +\frac{\eps^2}{2}\Delta_\eta w^\eps =
  \eps\(G\(u^\eps,\dots,u^\eps\)- 
G\(u_{\rm app}^\eps,\dots,u_{\rm app}^\eps\)\) 
+\l \eps r_1^\eps+\eps r_2^\eps+\eps r_3^\eps,
\end{equation*}
where
\begin{align*}
G(u_1,\dots,u_{2\nu+1})= 
\l\(K\ast(u_1\overline u_2\dots\overline u_{2\nu})\)u_{2\nu+1} 
+ \mu u_1\overline u_2 \dots u_{2\nu+1},\\
\end{align*}
and the remainder terms are given by
\begin{align*}
& r_1^\eps= 
\sum_{{(\ell_1,\dots,\ell_{2\nu+1})\in I_j} ,
\atop{ \ell_{2\nu+1}\not =j}} 
\Big( K\ast \( a_{\ell_1} \overline a_{\ell_2} \dots 
\overline a_{\ell_{2\nu}} 
e^{i(\phi_{\ell_1}-\phi_{\ell_2}\dots-\phi_{\ell_{2\nu}})/\eps} \) 
a_{\ell_{2\nu+1}} e^{i\phi_{\ell_{2\nu+1}}/\eps} \\
& \qquad\qquad\qquad\qquad\qquad -\widehat K(\k_j-\k_{\ell_{2\nu+1}}) 
a_{\ell_1} \overline a_{\ell_2} \dots a_{\ell_{2\nu+1}} 
e^{i\phi_j/\eps} \Big) , \\
& r_2^\eps=G\(u_{\rm app}^\eps,\dots,u_{\rm app}^\eps\)-
\sum_{j\in J}\sum_{(\ell_1,\dots,\ell_{2\nu+1})\in I_j}
G\( a_{\ell_1} e^{i\phi_{\ell_1}/\eps},\dots, 
a_{\ell_{2\nu+1}} e^{i\phi_{\ell_{2\nu+1}}/\eps} \) ,\\
& r_3^\eps=-\frac{\eps}{2}\sum_{j\in J}
e^{i\phi_j/\eps}\Delta_\eta a_j.
\end{align*}
The term $r_1^\eps$ corresponds to the stationary phase argument
performed formally in Section~\ref{sec:nonloc}, and is proved to be
$o(1)$ in
Section~\ref{sec:localizing}. The term $r_2^\eps$ is more standard,
and corresponds to the error introduced by non-resonant phases. It is
proved to be $\O(\eps)$ in Section~\ref{sec:nonstat} by a
non-stationary phase argument. Finally, the term $r_3^\eps$ corresponds
to the fact that the approximate solution was constructed by
canceling the $\O(1)$ and $\O(\eps)$ terms only in the formal WKB
construction: $r_3^\eps$ corresponds to the remaining $\O(\eps^2)$
terms, and is (rather obviously) of order $\O(\eps)$. 
\subsection{Localizing the non-local oscillations}
\label{sec:localizing}

\begin{lemma}\label{lem:phazstat}
Under the assumptions of Theorem~\ref{theo:wnlgo}, we have
  \begin{equation*}
    \|r_1^\eps\|_{L^\infty([0,T];W\cap L^2)}\Tend \eps 0 0.
  \end{equation*}
\end{lemma}
\begin{proof}
Let $j\in J$, $(\ell_1,\dots,\ell_{2\nu+1})\in I_j$,  
with $\kappa_{\ell_{2\nu+1}}\not=\k_j$. Denote 
\begin{align*}
A &= a_{\ell_1}\overline a_{\ell_2}\dots\overline a_{\ell_{2\nu}} , 
\quad a = a_{\ell_{2\nu+1}} , \quad \k = \k_{\ell_{2\nu+1}} , \\
b^\eps(x) &=E\(A(x) e^{i(\kappa_j-\kappa)\cdot x/\eps}\) 
a(x)e^{i\kappa\cdot x/\eps},\\
b^\eps_{\rm app}(x)&=\widehat K\(\kappa_j-\kappa\) 
A(x) a(x) e^{i\k_j\cdot x/\eps}, 
\end{align*}
where we have dropped the dependence upon 
$j,\ell_1,\dots,\ell_{2\nu+1}$ and $t$. Then
\begin{equation*}
r_1^\eps = \sum_{j\in J}\sum_{{(\ell_1,\dots,\ell_{2\nu+1})\in I_j},
\atop{ \ell_{2\nu+1}\not =j}} \(b^\eps -b^\eps_{\rm app}\) 
e^{it\partial_t\phi_j/\eps}, 
\end{equation*}
where the notation $t\partial_t\phi_j$ is there only to avoid a
discussion on the $\eta$'s.
We need to estimate $r_1^\eps$ in $W=\F(L^1)$ and $L^2=\F(L^2)$, so we compute:
\begin{align*}
  \widehat b^\eps(\xi) &= 
\F\( E\(A e^{i(\kappa_j-\kappa)\cdot x/\eps}\) 
a e^{i\kappa\cdot x/\eps} \) \( \xi\) \\
&= (2\pi)^{-d/2} \( \F\( E\(A e^{i(\kappa_j-\kappa)\cdot x/\eps}\) \) 
\ast \F\(a e^{i\kappa\cdot x/\eps} \) \) \( \xi\)\\
& =(2\pi)^{-d/2} \int \widehat K(\zeta) 
\widehat A \(\zeta-\frac{\kappa_j-\kappa}{\eps}\)
\widehat a\(\xi-\zeta-\frac{\kappa}{\eps}\)d\zeta\\
&= (2\pi)^{-d/2} \int \widehat K\(\zeta+\frac{\k_j-\k}{\eps}\)
\widehat A \(\zeta \)
\widehat a\(\xi-\zeta-\frac{\kappa_j}{\eps}\)d\zeta.
\end{align*}
On the other hand,
\begin{equation*}
 \widehat b^\eps_{\rm app}(\xi) =
 (2\pi)^{-d/2} \int \widehat K\(\k_j-\k\) \widehat A \(\zeta \)
\widehat a\(\xi-\zeta-\frac{\k_j}{\eps}\)d\zeta.
\end{equation*}
Since $\widehat K$ is homogeneous of degree zero, we infer
\begin{align*}
\widehat b^\eps(\xi) \, - &\,  \widehat b^\eps_{\rm app}(\xi) = \\
& (2\pi)^{-d/2} \int 
\( \widehat K\(\k_j-\k+\eps\zeta\)-\widehat K\(\k_j-\k\)\) 
\widehat A \(\zeta \)
\widehat a\(\xi-\zeta-\frac{\k_j}{\eps}\)d\zeta.
\end{align*}
Therefore,
\begin{align*}
\|r_1^\eps\|_{W\cap L^2}\le 
\sum_{j\in J}\sum_{{(\ell_1,\dots,\ell_{2\nu+1})\in I_j},
\atop{ \ell_{2\nu+1}\not =j}} \int_{\R^2}
& \left\lvert 
\widehat K \(\k_j-\k_{\ell_{2\nu+1}}+\eps\zeta\)- 
\widehat K \(\k_j-\k_{\ell_{2\nu+1}}\) 
\right\rvert \\
& \quad \left\lvert
\F\( a_{\ell_1}\overline a_{\ell_2}\dots\overline a_{\ell_{2\nu}} \) \(\zeta \)
\right\rvert
\left\lVert a_{\ell_{2\nu+1}} \right\rVert_{W\cap L^2}
d\zeta.
\end{align*}
To conclude, we note that $\widehat K\in L^\infty(\R^2)$, and
$\widehat K$ is continuous at $\k_j-\k_{\ell_{2\nu+1}}\not =0$. 
We can then conclude by the Dominated Convergence Theorem. 
\end{proof}
\begin{remark}
  The proof shows that, in general, we cannot expect a rate in our
  asymptotic error estimate. For 
  instance, for the non-local interaction in (DS), if $\k=(p,0)$, $p\not =0$,
  \begin{equation*}
  \widehat  K\(\k+\zeta\eps\)- \widehat K\(\k\)  =
  \frac{-\eps^2\zeta_2^2 }{(p+\eps \zeta_1)^2 +\eps^2\zeta_2^2 }.
  \end{equation*}
There is no \emph{uniform} control  (in $\zeta$) other than
\begin{equation*}
 \left\lvert  \widehat  K\(\k+\zeta\eps\)- \widehat
   K\(\k\)\right\rvert \le 1. 
\end{equation*}
\end{remark}
\subsection{Filtering the non-characteristic oscillations}
\label{sec:nonstat} 
Nonlinear interactions not only produce resonances, but also other
non-characteristic high frequency oscillations. The latter have to be
filtered via a non-stationary phase type argument. 
This becomes clear on the integral formulation for
\begin{equation*}
  i\eps\d_t w^\eps +\frac{\eps^2}{2}\Delta_\eta w^\eps = F^\eps,
\end{equation*}
which reads:
\begin{equation*}
  w^\eps(t,x)=U^\eps(t)w^\eps(0,x)-
  i\eps^{-1}\int_0^tU^\eps(t-\tau)F^\eps(\tau,x)d\tau. 
\end{equation*}
The main result of this paragraph is:
\begin{proposition}\label{lem:estR2}
  Let $(\alpha_j)_{j\in J}\in X^2(\R^d)$. Denote
  \begin{equation*}
    R_2^\eps(t,x) = -i \int_0^tU^\eps(t-\tau)r_2^\eps(\tau,x)d\tau.
  \end{equation*}
There exists $C>0$ such that
  for all $\eps\in ]0,1]$,
  \begin{equation*}
    \sup_{t\in [0,T]}\|R_2^\eps(t)\|_{W\cap L^2(\R^d)}\le C\eps. 
  \end{equation*}
\end{proposition}
To prove this result, we first reduce the analysis to the case of a
single oscillation. Decompose $r_2^\eps$ as
\begin{align*}
r_2^\eps &= G\(\sum_{\ell_1\in J}a_{\ell_1} e^{i\phi_{\ell_1}/\eps},
\dots , \sum_{\ell_{2\nu+1}\in J}
a_{\ell_{2\nu+1}} e^{i\phi_{\ell_{2\nu+1}}/\eps}\) \\
&\quad
- \sum_{j\in J}\sum_{(\ell_1,\dots,\ell_{2\nu+1})\in I_j}
G\( a_{\ell_1} e^{i\phi_{\ell_1}/\eps}, \dots,
a_{\ell_{2\nu+1}} e^{i\phi_{\ell_{2\nu+1}}/\eps}\) \\
&= \sum_{\ell_1,\dots,\ell_{2\nu+1}\in J}
G\( a_{\ell_1} e^{i\phi_{\ell_1}/\eps}, \dots ,
a_{\ell_{2\nu+1}} e^{i\phi_{\ell_{2\nu+1}}/\eps} \) \\
&\quad
- \sum_{j\in J}\sum_{(\ell_1,\dots,\ell_{2\nu+1})\in I_j}
G\( a_{\ell_1} e^{i\phi_{\ell_1}/\eps}, \dots,
a_{\ell_{2\nu+1}} e^{i\phi_{\ell_{2\nu+1}}/\eps} \) \\
&=\sum_{(\ell_1,\dots,\ell_{2\nu+1})\in N}
G\( a_{\ell_1} e^{i\phi_{\ell_1}/\eps}, \dots,
a_{\ell_{2\nu+1}} e^{i\phi_{\ell_{2\nu+1}}/\eps} \) ,
\end{align*}
where 
\begin{equation*}
  N= J^{2\nu+1}\setminus \bigcup_{j\in J}I_j
\end{equation*}
is the \emph{non-resonant set}. Write
\begin{align*}
G & \( a_{\ell_1} e^{i\phi_{\ell_1}/\eps}, \dots,
a_{\ell_{2\nu+1}} e^{i\phi_{\ell_{2\nu+1}}/\eps} \)= \\
& \l E \( a_{\ell_1} \overline a_{\ell_2} \dots 
\overline a_{\ell_{2\nu}}  
e^{i(\phi_{\ell_1}-\phi_{\ell_2}\dots-\phi_{\ell_{2\nu}})/\eps} \)
a_{\ell_{2\nu+1}} e^{i\phi_{\ell_{2\nu+1}}/\eps} \\ 
& + \mu a_{\ell_1}\dots a_{\ell_{2\nu+1}}
e^{i(\phi_{\ell_1}-\phi_{\ell_2}\dots+\phi_{\ell_{2\nu+1}})/\eps},
\end{align*}
and separate the temporal and spatial oscillations. 
The non-local term reads  
\begin{align*}
\exp \( -i\sum_{m=1}^d \eta_m \( |\k_{\ell_1,m}|^2-|\k_{\ell_2,m}|^2
\dots-|\k_{\ell_{2\nu},m}|^2 \) t/(2\eps) \) \times & \\ \times
E \( a_{\ell_1} \overline a_{\ell_2}\dots\overline a_{\ell_{2\nu}} 
e^{i(\k_{\ell_1}-\k_{\ell_2}\dots-\k_{\ell_{2\nu}})\cdot x/\eps} \) 
& a_{\ell_{2\nu+1}} e^{i\k_{\ell_{2\nu+1}}\cdot x/\eps},
\end{align*}
with, since $(\ell_1,\dots,\ell_{2\nu+1})\in N$,
\begin{equation*}
\sum_{m=1}^d \eta_m \( |\k_{\ell_1,m}|^2-|\k_{\ell_2,m}|^2
\dots-|\k_{\ell_{2\nu},m}|^2 \) \not =
\sum_{m=1}^d \eta_m |\k_{\ell_{2\nu+1},m}|^2 . 
\end{equation*}
We see that the following lemma is the key:
\begin{lemma}\label{lem:duhamelR}
 Let $T>0$, $\om\in\R$, $\k_1,\k_2\in\R^d$, and $b_1,b_2\in L^\infty([0,T];W\cap
 L^2(\R^d))$. Denote 
  \begin{equation*}
   D^\eps(t,x):= \int_0^t U^\eps(t-\tau)\(
   E\(b_1(\tau,x)e^{i\k_1\cdot x/\eps}\) b_2(\tau,x)e^{i\k_2\cdot
     x/\eps}   e^{i\om
      \tau/(2\eps)}\)d\tau .
  \end{equation*}
Let $\k=\k_1+\k_2$.  Assume $\om \not =|\k|^2$, and $\d_t b_j
,\Delta b_j \in 
L^\infty([0,T];W\cap 
 L^2(\R^d))$, $j=1,2$. Then
  \begin{align*}
    \lVert D^\eps\rVert_{X_T}\le
     \frac{C\eps}{\left\lvert |\k|^2-\om\right\rvert} \Big( & \<\k\>^2 
\left\lVert
      b_1b_2\right\rVert_{X_T}+ \left\lVert
      b_1\Delta b_2\right\rVert_{X_T}  
+ \left\lVert b_2\Delta b_1\right\rVert_{X_T}
+ \left\lVert \nabla b_1\nabla b_2\right\rVert_{X_T}\\
& +
\left\lVert b_1\d_t b_2\right\rVert_{X_T}+
\left\lVert b_2\d_t b_1\right\rVert_{X_T}\Big),
  \end{align*} 
where $\|f\|_{X_T}:= \|f\|_{L^\infty([0,T];W\cap L^2)}$, and 
 $C$ is independent of $\k_j$, $\om$ and $b_j$. 
\end{lemma}
\begin{proof}
Let 
\begin{equation*}
  f^\eps(t,x) =E\(b_1(t,x)e^{i\k_1\cdot x/\eps}\) b_2(t,x)e^{i\k_2\cdot
     x/\eps}   .
\end{equation*}
We compute, like in the proof of Lemma~\ref{lem:phazstat},
\begin{equation*}
  \widehat f^\eps(t,\xi) = (2\pi)^{-d/2}\(\(\widehat K\(\cdot
  +\frac{\k_1}{\eps}\) \widehat b_1(t,\cdot)\)\ast \widehat
  b_2(t,\cdot)\)\(\xi-\frac{\k}{\eps}\)=\widehat
  g^\eps\(t,\xi-\frac{\k}{\eps}\) , 
\end{equation*}
where
\begin{equation*}
  g^\eps= E^\eps (b_1)b_2,\quad\text{and}\quad
 \widehat{E^\eps(b)}(\xi)=\widehat K\(\xi
  +\frac{\k_1}{\eps}\) \widehat b(\xi) .
\end{equation*}
  By the definition of $U^\eps(t)$, we have
  \begin{equation*}
    \widehat D^\eps(t,\xi)= \int_0^t e^{-i\eps(t-\tau)|\xi|^2/2}
    \, \widehat g^\eps\(t,\xi-\frac{\k}{\eps}\)e^{-i\om
      \tau/(2\eps)}d\tau. 
  \end{equation*}
Setting $\eta = \xi -\k/\eps$, 
we have
 \begin{align*}
    \widehat D^\eps(t,\xi)&= e^{-i\eps t|\eta+\k/\eps|^2/2}\int_0^t
    e^{i\eps\tau|\eta+\k/\eps|^2/2} \,
    \widehat g^\eps\(\tau,\eta\)e^{-i\om
      \tau/(2\eps)}d\tau\\
&=e^{-i\eps t|\eta+\k/\eps|^2/2}\int_0^t
    e^{i\tau\theta/2} \,
    \widehat g^\eps\(\tau,\eta\)d\tau ,
  \end{align*}
where we have denoted 
\begin{equation*}
  \theta = \eps \left\lvert \eta+\frac{\k}{\eps}\right\rvert
  -\frac{\om}{\eps} = \underbrace{\eps|\eta|^2 +2\k\cdot
    \eta}_{\theta_1} +\underbrace{\frac{|\k|^2-\om}{\eps}}_{\theta_2}. 
\end{equation*}
Integrate by parts, by first integrating $e^{i\tau\theta_2/2}$: 
\begin{equation*}
  \widehat D^\eps(t,\xi) = -\frac{2i}{\theta_2} e^{i\tau\theta/2} 
    \widehat g^\eps\(\tau,\eta\)\Big|_0^t +\frac{2i}{\theta_2}\int_0^t
    e^{i\tau\theta/2} \( i\frac{\theta_1}{2}\widehat g^\eps \(\tau,\eta\)+
 {\d_t \widehat g^\eps}\(\tau,\eta\)\)d\tau.
\end{equation*}
The lemma follows, since $\widehat K\in L^\infty(\R^d)$.
\end{proof}
In view of Lemma~\ref{lem:morereg}, Proposition~\ref{lem:estR2} follows by
summation in Lemma~\ref{lem:duhamelR}. 

\begin{remark} Lemma \ref{lem:duhamelR} remains true if $E$ is
  replaced by the identity operator. In this 
  case, we simply extend \cite[Lemma~5.7]{CDS10} from the $W$ setting
  to the $W\cap L^2$ setting, an extension which requires absolutely
  no novelty. 
\end{remark}

\subsection{Proof of Theorem~\ref{theo:wnlgo}}\label{sec:proofWNLGO}

 Lemma~\ref{lem:morereg} shows that under the assumptions of
 Theorem~\ref{theo:wnlgo}, we also have
 \begin{equation}
   \label{eq:r3small}
   \sup_{t\in [0,T]}\|r^\eps_3(t)\|_{W\cap L^2(\R^d)} \lesssim \eps. 
 \end{equation}
Duhamel's formula for the error term $w^\eps=u^\eps-u^\eps_{\rm app}$
reads
\begin{align*}
  w^\eps(t)& = -i \int_0^t U^\eps(t-\tau)\(G(u^\eps,\dots,u^\eps)-
  G(u^\eps_{\rm app},\dots,u^\eps_{\rm app})\)(\tau)d\tau\\
&\quad -i\int_0^t U^\eps(t-\tau)\(\l r_1^\eps+r_2^\eps
+r_3^\eps\)(\tau)d\tau.  
\end{align*}
We then proceed in two steps:
\begin{enumerate}
\item Prove that $w^\eps$ is small (as in Theorem~\ref{theo:wnlgo}) in
  $W$, by a semilinear analysis.
\item Infer that $w^\eps$ is small in $L^2(\R^d)$, by a ``linear''
  analysis. 
\end{enumerate}
We note the point-wise identity
\begin{equation}\label{eq:Gpoint}
\begin{aligned}
G(u^\eps,\dots,u^\eps)-
G(u^\eps_{\rm app},\dots,u^\eps_{\rm app})
& = \(\l K\ast|u^\eps|^{2\nu} +\mu |u^\eps|^{2\nu}\) w^\eps \\
+ \(\l K\ast\(|u^\eps|^{2\nu}- |u^\eps_{\rm app}|^{2\nu}\) \)
& u^\eps_{\rm app}+\mu \(|u^\eps|^{2\nu}-
|u^\eps_{\rm app}|^{2\nu}\)u^\eps_{\rm app}. 
\end{aligned}
\end{equation}
Since $\widehat K \in L^\infty$, we infer
\begin{align*}
\|G(u^\eps,\dots,u^\eps)-
G(u^\eps_{\rm app},\dots,u^\eps_{\rm app})\|_W 
& \lesssim \|u^\eps\|_W^{2\nu} \|w^\eps\|_W \\
& + \(\|u^\eps\|_W^{2\nu-1}+ \|u^\eps_{\rm app}\|_W^{2\nu-1}\)
\|w^\eps\|_W \|u^\eps_{\rm app}\|_W\\
&\lesssim \(\|u^\eps_{\rm app}\|_W^{2\nu} + \|w^\eps\|_W^{2\nu}\)
\|w^\eps\|_W, 
\end{align*}
where time is fixed. We know from Lemma~\ref{lem:existprof} that
$u^\eps_{\rm app}\in C([0,T],W)$, so there exists $C_0$ independent of
$\eps\in ]0,1]$ such that
\begin{equation*}
  \|u^\eps_{\rm app}(t)\|_W\le C_0,\quad \forall t\in [0,T]. 
\end{equation*}
Since $u^\eps\in C([0,T^\eps],W)$ and $w^\eps_{\mid t=0}=0$, there
exists $t^\eps>0$ such that
\begin{equation}\label{eq:solong}
   \|w^\eps(t)\|_W\le C_0
\end{equation}
for $t\in [0,t^\eps]$. So long as \eqref{eq:solong} holds, we infer
\begin{equation*}
  \|w^\eps(t)\|_W \lesssim \int_0^t\|w^\eps(\tau)\|_W d\tau + |\l|o(1)
  + \eps,
\end{equation*}
where we have used Lemma~\ref{lem:phazstat}, Lemma~\ref{lem:estR2},
and \eqref{eq:r3small}. Gronwall lemma implies that so long as
\eqref{eq:solong} holds, 
\begin{equation*}
 \|w^\eps(t)\|_W \lesssim |\l|o(1)
  + \eps, 
\end{equation*}
where the right hand side does not depend on $t\in [0,T]$. Choosing
$\eps\in ]0,\eps_0]$ with $\eps_0$ sufficiently small, we see that
\eqref{eq:solong} remains true for $t\in [0,T]$, and the Wiener part
of Theorem~\ref{theo:wnlgo} follows. 
\smallbreak

For the $L^2$ setting, we resume \eqref{eq:Gpoint}. Plancherel's
identity and Young's inequality yield
\begin{align*}
\|G(u^\eps,\dots,u^\eps)-
G(u^\eps_{\rm app},\dots,u^\eps_{\rm app})\|_{L^2} 
& \lesssim \|u^\eps\|_W^{2\nu} \|w^\eps\|_{L^2} \\
& + \(\|u^\eps\|_W^{2\nu-1}+\|u^\eps_{\rm app}\|_{W}^{2\nu-1}\)
\|w^\eps\|_W \|u^\eps_{\rm app}\|_{L^2}. 
\end{align*}
By Lemma~\ref{lem:existprof},
$u^\eps_{\rm app}\in C([0,T],L^2(\R^d))$, so by the first part of the
proof of Theorem~\ref{theo:wnlgo}, the last line in the above
inequality is $\l o(1)+\O(\eps)$. We also know 
\begin{equation*}
  \|u^\eps(t)\|_W\le  2 C_0,\quad \forall t\in [0,T],
\end{equation*}
provided $\eps$ is sufficiently small. Gronwall lemma then shows
directly the estimate
\begin{equation*}
 \sup_{t\in [0,T]} \|w^\eps(t)\|_{L^2} \lesssim |\l|o(1)
  + \eps.
\end{equation*}
This completes the proof of Theorem~\ref{theo:wnlgo}. 

\smallbreak

Note that for $\l=0$, we get the rate $\O(\eps)$ for the remainder
term, while for $\l\not =0$, no rate is expected: this follows from
the analysis in \S\ref{sec:localizing}.

\section{More weakly nonlinear geometric optics}
\label{sec:moreweakly}

In this paragraph, we aim to get further insight on the geometric
optics approximation in Sobolev spaces of negative order. As we shall
see, estimates of the approximate solution (in negative order Sobolev
spaces) can  be somewhat counter-intuitive. 
To this end, we consider 
\begin{equation}
  \label{eq:moreweakly}
  i\eps\d_t u^\eps +\frac{\eps^2}{2}\Delta u^\eps = \mu \eps^J
  |u^\eps|^{2\nu} u^\eps \quad ,\quad u^\eps(0,x)=\sum_{j\in
    J_0}\alpha_j(x)e^{i\k_j\cdot x/\eps}. 
\end{equation} 
The regime $J=1$ is critical as far as nonlinear effects at leading
order are considered, according to \cite{CaBook}. For $J>1$, nonlinear
effects are negligible at leading order in $L^2\cap L^\infty$. We
shall analyze this phenomenon more precisely.
\subsection{Approximate solution}
\label{sec:approxweaker}
Pretending that even if $J>1$, the nonlinearity behaves like in the
critical case $J=1$, we can resume the discussion from
Section~\ref{sec:NLSWNLGO}: we consider the same resonant set, and the
transport system becomes
\begin{equation*}
  \d_t a_j^\eps + \k_j\cdot \nabla a_j^\eps = -i\mu
  \eps^{J-1}\sum_{(\ell_1,\dots,\ell_{2\nu+1})\in I_j} a_{\ell_1}^\eps\overline
  a_{\ell_2}^\eps\dots a_{\ell_{2\nu+1}}^\eps\quad ,\quad a^\eps_{j\mid
    t=0}=\alpha_j, 
\end{equation*}
where the notation now emphasizes that the presence of $\eps$ in the
equation makes the profiles $\eps$-dependent. 
Working in the same functional framework as in Section~\ref{sec:construct}, we
construct profiles, for which we prove first
\begin{equation*}
  (a_j^\eps)_{j\in J}\in C([0,T],X(\R^d))
\end{equation*}
for some $T>0$, uniformly in $\eps\in [0,1]$, then infer
\begin{equation*}
  a_j^\eps(t,x) = \alpha_j(x-t\k_j) + \O\(\eps^{J-1}\) \text{ in }C([0,T],W\cap
  L^2(\R^d)). 
\end{equation*}
Setting
\begin{equation*}
  u_{\rm app}^\eps(t,x) = \sum_{j\in
    J}a_j^\eps(t,x)e^{i\phi_j(t,x)/\eps}, 
\end{equation*}
a straightforward adaptation of Theorem~\ref{theo:wnlgo} yields,
provided that we start with suitable initial profiles,
\begin{equation*}
 \sup_{t\in [0,T]} \|u^\eps(t)-u_{\rm app}^\eps(t)\|_{W\cap
   L^2}=\O\(\eps\). 
\end{equation*}
\subsection{Negligible or not?}
\label{sec:neglornot}
In view of the proof of the norm
inflation phenomenon, we shall now focus on the case of
Example~\ref{ex:key}. We know from before, that, starting with three
non-trivial $\eps$-oscillations, the zero 
mode instantaneously appears at order $\eps^{J-1}$. 
For future reference, we prove a result whose assumptions will become
clear later on.
\begin{lemma}\label{lem:estimate}
  Let $d\ge 1$, $\beta>0$. For $f\in \Sch'(\R^d)$ and $\kappa\in \R^d$, 
  we denote
  \begin{equation*}
    I^\eps(f,\kappa)(x)= f\(x
  \eps^{(1-\beta)/2}\) e^{i\kappa\cdot x/\eps^{(1+\beta)/2}}
  \end{equation*}
$(1)$ Let $\kappa\in \R^d$, with $\kappa\not=0$. For all $\si\le 0$,
there exists $C=C(\sigma,\kappa)$ such that for all $f\in\Sch(\R^d)$,   
\begin{equation*}
   \|I^\eps(f,\kappa)\|_{H^\si(\R^d)}^2 \le
   C\eps^{-d(1-\beta)/2+(1+\beta)|\si|}
\|f\|^2_{H^m(\R^d)},   
\end{equation*}
with
\begin{itemize}
\item $m=|\si|$ if $\beta\le 1$
\item $m= \(\frac{1+\beta}{2}\)|\si|$ if $\beta\ge 1$. 
\end{itemize}
In addition, we have $C(\si,\kappa)\to 0$ as $|\kappa|\to +\infty$. \\
$(2)$ For all $\si\le 0$, $\beta<1$ and $f\in
L^2(\R^d)$,
\begin{equation*}
  \|I^\eps(f,0)\|_{H^\si(\R^d)}^2 = \eps^{ -d(1-\beta)/2}
\(\|f\|_{L^2(\R^d)}^2+o(1)\), \quad \text{as }\eps \to 0.
\end{equation*}
$(3)$ 
If $\beta =1$, $\si \in \R$ and $f\in H^\si(\R^d)$, 
$  \|I^\eps(f,0)\|_{H^\si(\R^d)}^2 = 
\|f\|_{H^\si(\R^d)}^2.$\\
$(4)$ 
If $\beta>1$, $\si\le 0$, and $f\in H^\si(\R^d)$, 
\begin{equation*}
 \|I^\eps(f,0)\|_{H^\si(\R^d)}^2 \ge
 \eps^{-d(1-\beta)/2+(\beta-1)|\si|}  \|f\|_{H^\si(\R^d)}^2.
\end{equation*}
\end{lemma}
  This result shows in particular that for $\si\le 0$, $\kappa\not =0$
  and $f$ sufficiently smooth, we always have
  $$\|I^\eps(f,0)\|_{H^\si(\R^d)}\gg
  \|I^\eps(f,\kappa)\|_{H^\si(\R^d)}.$$
\begin{proof}
  We compute
\begin{align*}
  \widehat{I^\eps(f,\kappa)}(\xi) &= 
  \frac{1}{(2\pi)^{d/2}}\int 
  e^{-ix\cdot \xi} f\(x
  \eps^{(1-\beta)/2}\) e^{i\kappa\cdot x/\eps^{(1+\beta)/2}}dx \\
&=
  \eps^{- d (1-\beta)/2} \frac{1}{(2\pi)^{d/2}}\int 
  e^{-iy\cdot \xi/\eps^{(1-\beta)/2} } f\(y\) e^{i\kappa\cdot
    y/\eps}dy  \\
&= \eps^{- d (1-\beta)/2}\widehat f\(
\frac{\xi}{\eps^{(1-\beta)/2}}- \frac{\kappa}{\eps}\). 
\end{align*}
Therefore, 
\begin{align*}
  \|I^\eps(f,\kappa)\|_{H^\si(\R^d)}^2 & = \int \<\xi\>^{2\si} \left\lvert
 \widehat{I^\eps(f,\kappa)} (\xi)\right\rvert^2d\xi \\ 
& = \eps^{-d(1-\beta)}\int \<\xi\>^{2\si} \left\lvert
\widehat f\(
\frac{\xi}{\eps^{(1-\beta)/2}}- \frac{\kappa}{\eps}\)\right\rvert^2d\xi.
\end{align*}
To prove the first point, we write, for $\si\le 0$, and $\beta\le 1$,
\begin{align*}
& \eps^{d(1-\beta)} \|I^\eps(f,\kappa)\|_{H^\si(\R^d)}^2= \\
&= \int \<\xi\>^{2\si} 
\<\frac{\xi}{\eps^{(1-\beta)/2}}- \frac{\kappa}{\eps}\>^{2\si}
\<\frac{\xi}{\eps^{(1-\beta)/2}}- \frac{\kappa}{\eps}\>^{2|\si|}
\left\lvert 
\widehat f\(
\frac{\xi}{\eps^{(1-\beta)/2}}-
\frac{\kappa}{\eps}\)\right\rvert^2d\xi\\
&\le \sup_{\xi \in \R^d}\(\<\xi\>^{-1}
\<\frac{\xi}{\eps^{(1-\beta)/2}}-
\frac{\kappa}{\eps}\>^{-1}\)^{2|\si|}\eps^{d(1-\beta)/2}\|f\|^2_{H^{|\si|}(\R^d)} .
\end{align*}
Next, write 
\begin{align*}
\<\frac{\xi}{\eps^{(1-\beta)/2}}-\frac{\kappa}{\eps}\>^{-1} 
= \< \eps^{(\beta-1)/2} 
\left( \xi - \frac{\kappa}{\eps^{(1+\beta)/2}} \right) \>^{-1} 
 \le \< \xi - \frac{\kappa}{\eps^{(1+\beta)/2}} \>^{-1},
\end{align*}
where we have used the assumption $\beta\le 1$. 
Then use Peetre inequality (see e.g. \cite{Treves1}) 
to get the desired estimate in the case $\beta\le 1$. 

In the case $\beta>1$, we use another decomposition:
\begin{align*}
& \eps^{d(1-\beta)} \|I^\eps(f,\kappa)\|_{H^\si(\R^d)}^2= \\
&= \int \<\xi\>^{2\si} 
\<\frac{\xi}{\eps^{(1-\beta)/2}}- \frac{\kappa}{\eps}\>^{(1+\beta)\si}
\<\frac{\xi}{\eps^{(1-\beta)/2}}- \frac{\kappa}{\eps}\>^{(1+\beta)|\si|}
\left\lvert 
\widehat f\(
\frac{\xi}{\eps^{(1-\beta)/2}}-
\frac{\kappa}{\eps}\)\right\rvert^2d\xi\\
&\le \sup_{\xi \in \R^d}\(\<\xi\>^{-2}
\<\frac{\xi}{\eps^{(1-\beta)/2}}-
\frac{\kappa}{\eps}\>^{-(1+\beta)}\)^{|\si|}\eps^{d(1-\beta)/2}
\|f\|^2_{H^{(1+\beta)|\si|/2}(\R^d)}. 
\end{align*}
We use the obvious estimate
\begin{equation*}
  \<\xi\>^2 \<\frac{\xi}{\eps^{(1-\beta)/2}}-
\frac{\kappa}{\eps}\>^{1+\beta}\gtrsim
\left\{
  \begin{aligned}
  \< \frac{\kappa}{\eps}\>^{1+\beta} & \text{ if } |\xi| \le
  |\kappa|/2\eps^{(1+\beta)/2},\\ 
\<\frac{\kappa}{\eps^{(1+\beta)/2}}\>^2& \text{ if } |\xi| \ge
|\kappa|/2\eps^{(1+\beta)/2} .
  \end{aligned}
\right.
\end{equation*}
In both cases, we infer
\begin{equation*}
  \<\xi\>^2 \<\frac{\xi}{\eps^{(1-\beta)/2}}-
\frac{\kappa}{\eps}\>^{1+\beta}\gtrsim \eps^{-(1+\beta)},
\end{equation*}
which yields the first point of the lemma. To prove the second point,
write 
\begin{align*}
 \|I^\eps(f,0)\|_{H^\si(\R^d)}^2 & =  \eps^{-d(1-\beta)}
\int \<\xi\>^{2\si} \left\lvert 
\widehat f\(
\frac{\xi}{\eps^{(1-\beta)/2}}\)\right\rvert^2d\xi\\
&=\eps^{-d(1-\beta)/2}\int \<\eps^{(1-\beta)/2}\xi\>^{2\si} \left\lvert 
\widehat f\(
\xi\)\right\rvert^2d\xi.
\end{align*}
In the case $\beta<1$, we conclude thanks to the Dominated Convergence
Theorem. 
The third point of the lemma ($\beta=1$) is obvious.  
To prove the last point, we write, 
\begin{align*}
 \|I^\eps(f,0)\|_{H^\si(\R^d)}^2 & =\eps^{-d(1-\beta)/2}
\int \<\eps^{(1-\beta)/2}\xi\>^{2\si} \left\lvert 
\widehat f\(
\xi\)\right\rvert^2d\xi\\
& =\eps^{-d(1-\beta)/2}
\int \frac{1}{\(1 +\eps^{1-\beta}|\xi|^2\)^{|\si|}} \left\lvert 
\widehat f\(
\xi\)\right\rvert^2d\xi\\
& =\eps^{-d(1-\beta)/2}\int \frac{\eps^{(\beta-1)|\si|}}{\(\eps^{\beta-1}
   +|\xi|^2\)^{|\si|}} \left\lvert \widehat
   f\(\xi\)\right\rvert^2d\xi\\
& \ge \eps^{-d(1-\beta)/2+ (\beta-1)|\si|}\int \frac{1}{\(1
   +|\xi|^2\)^{|\si|}} \left\lvert \widehat
   f\(\xi\)\right\rvert^2d\xi,
\end{align*}
and the result follows.
\end{proof}
\begin{remark}\label{rem:faiche}
  The last estimate of Lemma~\ref{lem:estimate} is sharp in terms of
  power of $\eps$, since by dominated convergence
  \begin{equation*}
    \|I^\eps(f,0)\|_{H^\si(\R^d)}^2 \Eq \eps 0 
\eps^{-d(1-\beta)/2+ (\beta-1)|\si|}\int_{\R^d} \frac{1}{
   |\xi|^{2|\si|}} \left\lvert \widehat
   f\(\xi\)\right\rvert^2d\xi,
    \end{equation*}
  for all $f\in L^2\cap H^\sigma$ when $-d/2<\si<0$, and 
  for all $f\in L^2\cap H^\sigma$ such that 
  $0\notin\text{supp}\widehat f$ when $\si\le-d/2$.
\end{remark}

Next, we shall simply apply Lemma~\ref{lem:estimate}  (in the case  
$\beta=1$) to $u^\eps_{\rm  app}$. We find, thanks to
Lemma~\ref{lem:a0NLS}, 
\begin{equation*}
  \|a^\eps_0(t)\|_{H^s(\R^d)} \approx \eps^{J-1},
\end{equation*}
for $t>0$ arbitrarily small. 
On the other hand, the first point of
Lemma~\ref{lem:estimate}  yields, for $s\le 0$:
\begin{equation*}
  \|u^\eps_{\rm app}(t)- a^\eps_0(t)\|_{H^s(\R^d)} \lesssim
  \eps^{|s|}. 
\end{equation*}
We infer, if $s\le 0$,
\begin{equation*}
  \|u^\eps(t)\|_{H^s(\R^d)} = \|a^\eps_0(t)\|_{H^s(\R^d)}  +
  \O\(\eps^{|s|}\)+\O\(\eps\),
\end{equation*}
where the last term stems from the geometric optics approximation, and
the simple control, for $s\le 0$, $\|f\|_{H^s}\le \|f\|_{L^2}$. We
conclude that for $t>0$ arbitrarily small, the zero mode \emph{is not 
negligible} in $H^s(\R^d)$, provided
\begin{equation*}
  |s|>J-1 \text{ and }J-1<1, \text{ that is } s<1-J<0 \text{ and }J<2.
\end{equation*}
We finally remark, that having the zero mode not negligible at leading order
means that nonlinear effects are present at leading order, in
$H^s(\R^d)$. We summarize these remarks in the following
\begin{proposition}
Let $s,J\in\R$ satisfy $s<1-J<0$ and $J<2$. Set $J_0=\{1,2,3\}$, 
$\k_0=0_{\R^d}$, and consider $\Phi_0$ from \eqref{eq:J0}. 
Then, there exist $\alpha_1, \alpha_2, \alpha_3 \in \Sch(\R^d)$, independent of $s$ and $J$, and a $T>0$, such that the unique 
solution $u^\eps \in C([0,T],L^2\cap L^\infty)$ to 
\eqref{eq:moreweakly} satisfies for all $t\in ]0,T]$ where
$a_0^\eps(t)\not =0$: 
$$\|u^\eps(t)\|_{H^s(\R^d)} \Eq \eps 0 \|a^\eps_0(t)\|_{H^s(\R^d)} 
\approx \eps^{J-1} \gg \|u^\eps(0)\|_{H^s(\R^d)} \approx \eps^{|s|}
\text{ as } \eps\to 0.$$
\end{proposition}

\section{Norm inflation}
\label{sec:outline}

To explain our approach, we first consider the
nonlinear Schr\"odinger equation:
\begin{equation}
  \label{eq:NLSCauchy} 
  i\d_t \psi+\frac{1}{2}\Delta \psi = \mu |\psi|^{2\nu}\psi,\quad x\in
  \R^d\quad , \quad \psi_{\mid t=0}=\varphi. 
\end{equation}
We proceed in four steps:
\begin{enumerate}
\item Choice of a suitable scaling in order to be able to use weakly
  nonlinear geometric optics.
\item Link between the Sobolev norms of $\psi$ and approximate
  solutions given by geometric optics.
\item High frequency analysis (WNLGO).
\item Conclusion: what WNLGO implies in terms of $\psi$.
\end{enumerate}

\subsection{Scaling}
\label{sec:scaling}

 We consider the general scaling
\begin{equation*}
  u^\eps(t,x)=
  \eps^\alpha\psi\(\eps^\beta t,\eps^\gamma x\). 
\end{equation*}
To simplify the discussion, we want to fix $\alpha,\beta,\gamma$ so
that $\psi$ solves \eqref{eq:NLSCauchy} and $u^\eps$ solves 
\begin{equation*}
  i\eps \d_t u^\eps +\frac{\eps^2}{2}\Delta u^\eps = \mu \eps^J
  |u^\eps|^{2\nu} u^\eps,
\end{equation*}
with $1\le J<2$. 
We will relate phenomena affecting $u^\eps$ for times of order $\O(1)$
with a norm inflation for $\psi$ on times of order $o(1)$: this
imposes $\beta>0$. We find the
relation
\begin{equation*}
  1+\beta=2+2\gamma= J+2\nu\alpha. 
\end{equation*}
Leaving only $\beta$ as a free parameter, this means
\begin{equation}\label{eq:scaling}
 u^\eps(t,x)= \eps^{(\beta+1-J)/(2\nu)}
\psi\(\eps^\beta t,\eps^{(\beta-1)/2}x\). 
\end{equation}
The initial data that we want to consider
for $u^\eps$ are 
\begin{equation*}
  u^\eps(0,x)= \sum_{j\in J_0} \alpha_j\(x
  \) e^{i\kappa_j\cdot x/\eps},
\end{equation*}
with $\kappa_j\in \R^d$ and $\alpha_j \in \Sch(\R^d)$. 
In view of \eqref{eq:scaling}, this yields
\begin{equation*}
  \psi(0,x) = \eps^{-(\beta+1-J)/(2\nu)}\sum_{j\in J_0} \alpha_j\(x
  \eps^{(1-\beta)/2}\) e^{i\kappa_j\cdot x/\eps^{(1+\beta)/2}}. 
\end{equation*}
This is exactly the scaling used in Lemma~\ref{lem:estimate}, up to
the factor $\eps^{-(\beta+1-J)/(2\nu)}$.
\subsection{High frequency analysis}
\label{sec:high}
We resume the framework of Example~\ref{ex:key}, and suppose that at
time $t=0$, $u^\eps$ is the sum of three plane waves:
\begin{equation*}
  u^\eps(0,x) = \sum_{j=1}^3 \alpha_j(x)e^{i\k_j\cdot x/\eps}, 
\end{equation*}
with $\alpha_1,\alpha_2,\alpha_3\in \Sch(\R^d)$ and 
\begin{equation*}
  \k_1=(1,0,\dots,0), \ \k_2=(1,1,0,\dots,0),\
  \k_3=(0,1,0,\dots,0) \in \R^d.  
\end{equation*}
The important point is that by nonlinear resonance, the zero mode
appears (and possibly other modes):
\begin{equation*}
  u^\eps_{\rm app}(t,x) = a_0(t,x) + \sum_{j=1}^\infty
  a_j(t,x)e^{i\phi_j(t,x)/\eps},\quad \phi_j(t,x)= \k_j\cdot x-
  \frac{t}{2}|\k_j|^2,
\end{equation*}
where, for $j\ge 1$, we have $\k_j\in \Z^d\setminus\{0\}$, and 
the series is convergent in $L^2(\R^d)$, and more generally in all
Sobolev spaces from Remark~\ref{rem:forinflation}. 
By Lemma~\ref{lem:a0NLS} (or Lemma~\ref{lem:a0DS}, or 
Lemma~\ref{lem:a0genNLS}), even though $a_0$ is zero at time $t=0$, 
we can choose initial profiles so that 
$\d_t a_{0\mid t=0}\not =0$: this mode becomes instantaneously 
non-trivial. Geometric optics yields:
\begin{equation}\label{eq:ogok}
  \|u^\eps-u^\eps_{\rm app}\|_{L^\infty([0,T];L^2(\R^d))} 
  \Tend \eps 0 0 . 
\end{equation}
Below, we take advantage of this approximation, and of the fact that 
the new (non-oscillating) generated mode $a_0$ is much larger 
than the others in negative order Sobolev spaces, as  measured by
Lemma~\ref{lem:estimate}.  
\subsection{Proof of Theorem~\ref{theo:genNLSinflation}}
\label{sec:genNLSinflation}
In this case, we choose $J=1$. The reason why we have no flexibility
for $J$ here is that in Section~\ref{sec:moreweakly}, we have used the
fact that a rate for the error estimate is available,
$\|u^\eps-u^\eps_{\rm app}\|_{L^\infty([0,T],W\cap
  L^2)}=\O(\eps)$. Unlike the (NLS) case, no rate is available in
general in the presence of a non-local term; see Section~\ref{sec:localizing}. 
\smallbreak

For $\eps=1/n$, denote by $\psi_n$ the solution given by
\eqref{eq:scaling}, and by $\varphi_n$ its trace at
$t=0$. Lemma~\ref{lem:estimate} yields, for $s< 0$:
\begin{equation*}
  \|\varphi_n\|^2_{H^s(\R^d)}\lesssim \eps^{-\beta/\nu -d(1-\beta)/2
    +|s| (1+\beta)}. 
\end{equation*}
We have $\|\varphi_n\|_{H^s}\to 0 $ provided
\begin{equation}\label{eq:betacond1}
   -\frac{\beta}{\nu} -d\frac{1-\beta}{2}+|s|
  (1+\beta)>0\Longleftrightarrow \beta>\frac{d/2-|s|}{s_c+|s|},
\end{equation}
where $s_c$, given by \eqref{eq:sc},  is always non-negative in
the framework of this paper.
\smallbreak

Let $\tau>0$ independent of $\eps$ be such that $a_0(\tau)\not =0$. 
Set $t_n = \tau \eps^\beta= \tau/n^\beta$: $t_n\to 0$ provided
$\beta>0$. Denote by $\psi_{\rm app}$ the function obtained from
$u^\eps_{\rm app}$ via the scaling \eqref{eq:scaling} (the dependence
upon $n$ is omitted to ease the notation). 
Consider $\si\le 0$. We have obviously
\begin{equation*}
  \|\psi_n(t_n)- \psi_{\rm app}(t_n)\|_{H^\si(\R^d)}\le \|\psi_n(t_n)-
    \psi_{\rm app}(t_n)\|_{L^2(\R^d)} .
\end{equation*}
Estimate \eqref{eq:ogok} shows that we have
\begin{equation*}
  \|\psi_n(t_n)-
    \psi_{\rm app}(t_n)\|_{L^2(\R^d)} = o\( \|\psi_{\rm
      app}(t_n)\|_{L^2(\R^d)}\)\text{ as }n\to +\infty. 
\end{equation*}
We assume $0<\beta\le 1$. 
Lemma~\ref{lem:estimate} yields, for $\beta<1$,
\begin{equation*}
  \|\psi_{\rm app}(t_n)\|^2_{H^\si(\R^d)}\Eq \eps 0
  \|\psi_{\rm app}(t_n)\|^2_{L^2(\R^d)}\Eq \eps 0 \eps^{-\beta/\nu 
    -d(1-\beta)/2}\|a_0(\tau)\|_{L^2(\R^d)}^2. 
\end{equation*}
For $\beta =1$, we still have
\begin{equation*}
  \|\psi_{\rm app}(t_n)\|^2_{H^\si(\R^d)}\approx \eps^{-\beta/\nu 
    -d(1-\beta)/2} \approx
  \|\psi_{\rm app}(t_n)\|^2_{L^2(\R^d)}.
\end{equation*}
We infer, for $\beta\le 1$,
\begin{equation*}
  \|\psi_n(t_n)\|^2_{H^\si(\R^d)}\Eq \eps 0
  \|\psi_{\rm app}(t_n)\|^2_{H^\si(\R^d)}\approx \eps^{-\beta/\nu 
    -d(1-\beta)/2}. 
\end{equation*}
This power of $\eps$ is always negative, since we have $s_c\ge
0$, and
$\beta s_c\le s_c<d/2$. So to prove norm inflation, we simply have to
check the compatibility of \eqref{eq:betacond1} with the condition
$0<\beta\le 1$:
\begin{equation*}
  \frac{d/2-|s|}{s_c+|s|}<1\Longleftrightarrow  |s|>\frac{1}{2\nu}. 
\end{equation*}
The case of equality, which corresponds to the statement of
Theorem~\ref{theo:genNLSinflation}, can be reached thanks to logarithmic
modifications (multiply the initial data by $\ln \eps$), in the same
spirit as in \cite{CCT2,AlCa09}. 
\smallbreak

Finally, we simply note that all the negative order Sobolev norms of
$\psi$ 
become unbounded along the sequence of time $t_n$. It is then obvious
that so do the positive order Sobolev norms. 

\subsection{Proof of Theorem~\ref{theo:inflation2}}
\label{sec:inflation2}

We now assume $\l=0$: there is no non-local term, and we can use the
analysis of Section~\ref{sec:moreweakly}, with $1\le J<2$. We mimic
the discussion from the previous paragraph, concerning the algebraic
requirements on the different parameters: $\beta$, $s$, and now $J$. 
Lemma~\ref{lem:estimate} yields, for $s\le 0$:
\begin{equation*}
  \|\varphi_n\|^2_{H^s(\R^d)}\lesssim \eps^{-(\beta+1-J)/\nu -d(1-\beta)/2
    +|s| (1+\beta)}, 
\end{equation*}
so we demand
\begin{equation}\label{eq:beta22}
  \beta>\frac{d/2-|s|-(J-1)/\nu}{s_c+|s|}. 
\end{equation}
We will still demand $0<\beta\le 1$, so for all $\si\in \R$,
\begin{equation*}
  \|\psi_n(t_n)\|^2_{H^\si(\R^d)}\Eq \eps 0
  \|\psi_{\rm app}(t_n)\|^2_{H^\si(\R^d)}\approx \eps^{2(J-1)}
\eps^{-(\beta+1-J)/\nu 
    -d(1-\beta)/2}, 
\end{equation*}
where the new term $\eps^{2(J-1)}$ is due to the fact that we consider
``more weakly'' nonlinear geometric optics. 
This total power of $\eps$ is negative provided
\begin{equation}\label{eq:beta23}
  \beta s_c<\frac{d}{2} -(J-1)\(2+\frac{1}{\nu}\). 
\end{equation}
The  algebraic requirements are $0<\beta\le 1$, $1\le J<2$, 
\eqref{eq:beta22}, and \eqref{eq:beta23}. We check that they are
compatible, provided $s<-1/(1+2\nu)$. For such an $s$, we can find
$\delta>0$ so that
\begin{equation*}
  s = -\frac{\delta}{\nu} -\frac{1}{1+2\nu}. 
\end{equation*}
Pick $\beta=1$ and 
\begin{equation*}
  J = \frac{2+2\nu}{1+2\nu}-\delta.
\end{equation*}
The first two conditions are obviously fulfilled, at least if
$0<\delta\ll 1$ (it suffices to prove Theorem~\ref{theo:inflation2}
for $s$ close to $-1/(1+2\nu)$). A direct computation shows that so
are \eqref{eq:beta22} and
\eqref{eq:beta23}. Theorem~\ref{theo:inflation2} follows. 
\appendix

\section{Proof of Proposition~\ref{prop:ill}}
\label{sec:ill} 

Without recalling all details of \cite[Proposition~1]{BeTa05}, we shall give
a flavor of this rather general result, and explain how to infer
Proposition~\ref{prop:ill}. Roughly speaking, it suffices to prove that
one term in the Picard iteration process rules out
Definition~\ref{def:WP}, in order to deny well-posed for the solution to the
nonlinear problem. Therefore, we start with the free equation
\begin{equation*}
  i\d_t \psi + \frac{1}{2}\Delta_\eta \psi = 0\quad ,\quad \psi_{\mid
    t=0}=\varphi. 
\end{equation*}
For (NLS), we then consider the integral term
\begin{equation*}
  D(\varphi)(t,x)= -i\mu\int_0^t e^{i\frac{t-\tau}{2}\Delta_\eta}
  \(|\psi|^{2\nu}\psi\)(\tau,x)d\tau. 
\end{equation*}
To prove Proposition~\ref{prop:ill}, it suffices to show that the map
$\varphi\mapsto D(\varphi)$ is not continuous from $H^s(\R^d)$ to
$C([0,T],H^\si(\R^d))$, that is, there is no such control as
\begin{equation}\label{eq:cont}
  \|D(\varphi)\|_{L^\infty([0,T],H^\si(\R^d))}\lesssim
  \|\varphi\|^{2\nu+1}_{H^s(\R^d)} . 
\end{equation}
The main difference with the approach of Section~\ref{sec:outline} is
that now the analysis is ``much more linear''. In practice, we resume
the same lines as in  Section~\ref{sec:outline}, up to the factor
$\eps^{(\beta+1-J)/(2\nu)}$, which was there only to get precisely a
weakly nonlinear regime. We also fix $\beta=1$, and consider
\begin{equation*}
  \varphi(x) = \sum_{j=1}^3 \alpha_j(x) e^{i\k_j\cdot x/\eps},
\end{equation*}
where $\alpha_j\in \Sch(\R^d)$ and the $\k_j$'s are given by
Example~\ref{ex:key}. By creation of the zero mode (from 
Lemma~\ref{lem:a0NLS}), and Lemma~\ref{lem:estimate}, 
\eqref{eq:cont} would imply
\begin{equation*}
  1\lesssim \eps^{-s(2\nu+1)}, 
\end{equation*}
which is impossible if $s<0$. 
\smallbreak 

In the case where the non-local term is present, one can argue along
the same lines. For (DS), the assumption $\l+2\mu\not =0$ arises 
when one wants to use Lemma~\ref{lem:a0DS} in place of 
Lemma~\ref{lem:a0NLS}. For \eqref{eq:dipole}, one uses
Lemma~\ref{lem:a0nonlocNLS}.

\section{On negative
  order Sobolev spaces}
\label{sec:sobolev}
Lemma~\ref{lem:estimate} with $\beta=1$ shows that all Sobolev norms
for $I^\eps(f,\kappa)= f(x)e^{i\kappa\cdot x/\eps}$ behave according
to the intuition as $\eps\to 0$, provided $f\in \Sch(\R^d)$, that is
\begin{equation*}
\|I^\eps(f,\kappa)\|_{H^s(\R^d)}\lesssim \eps^{-s}, \quad \forall s\in
\R,\text{ if }\kappa\not =0.   
\end{equation*}
The aim of this appendix is to show that in general, negative order
Sobolev norms can behave rather strangely on functions which exhibit
rapid oscillations and/or concentration effects (as it is typically
the case for  
wave functions of quantum mechanics in the semi-classical limit)

\begin{example}[Oscillatory functions]
\label{sec:wkb}

We consider a WKB state with nonlinear phase function $\phi(x)=-\frac12
|x|^2$: 
  \begin{equation*}
    g^\eps(x) = e^{-|x|^2/2} e^{-i|x|^2/(2\eps)},\quad x\in \R^d.
  \end{equation*}

\begin{lemma}
  Let $d\ge 1$. Then 
\begin{equation*}
  \|g^\eps\|_{H^s(\R^d)} \approx 
\left\{
  \begin{aligned}
    \eps^{-s}&\text{ if } s>-d/2,\\
    \eps^{d/2}& \text{ if } s<-d/2.
  \end{aligned}
\right.
\end{equation*}
\end{lemma}
\begin{proof}
  Consider more generally, for $z\in \C$ with $\RE z>0$, 
  \begin{equation*}
    g_z(x)= e^{-z|x|^2/2}. 
  \end{equation*}
We compute:
\begin{equation*}
  \F g_z (\xi) = z^{-d/2} e^{-|\xi|^2/(2z)}.
\end{equation*}
For $s\in \R$, we have, if $z=a+ib$, $a,b\in \R$, $a>0$:
\begin{align}
\notag  \|g_z\|_{H^s(\R^d)}^2 &= \int_{\R^d} \<\xi\>^{2s}\left\lvert \F
    g_z(\xi)\right\rvert^2 d\xi=
  \frac{1}{|z|^d}\int_{\R^d}\<\xi\>^{2s}e^{-\frac{a}{a^2+b^2}
    |\xi|^2}d\xi\\ 
&=
\frac{1}{a^{d/2}}\int_{\R^d}\<\(\frac{a^2+b^2}{a}\)^{1/2}\eta\>^{2s}
e^{-|\eta|^2}d\eta. \label{eq:HsGauss}
\end{align}
In the present case, $z=1+i/\eps$:
\begin{align*}
  \|g^\eps\|_{H^s(\R^d)}^2 &=  \int_{\R^d}\<\(1+\frac{1}{\eps^2}\)^{1/2}\eta\>^{2s}
e^{-|\eta|^2}d\eta\\
&\approx \int_{\R^d}\<\frac{\eta}{\eps}\>^{2s}
e^{-|\eta|^2}d\eta= c(d)\int_0^{+\infty}\(1+\frac{r^2}{\eps^2}\)^s
e^{-r^2}r^{d-1}dr.
\end{align*}
We split the last integral into
$\int_0^\eps+\int_{\eps}^{+\infty}$. Then, we have
\begin{equation*}
  \int_0^\eps \(1+\frac{r^2}{\eps^2}\)^s
e^{-r^2}r^{d-1}dr \approx \eps^d,
\end{equation*}
and by examining the local integrability near zero, we find
\begin{equation*}
  \int_\eps^{+\infty}\(1+\frac{r^2}{\eps^2}\)^s
e^{-r^2}r^{d-1}dr \approx
\left\{
  \begin{aligned}
    \eps^{d}&\text{ if } s<-d/2,\\
 \eps^{-2s}&\text{ if } s>-d/2. 
  \end{aligned}
\right.
\end{equation*}
The lemma follows.
\end{proof}

\end{example}

\begin{example}[Concentrating functions]
Another important example concerns functions which concentrate at a point, e.g.
  \begin{equation*}
    p^\eps(x) = \eps^{-d/4}e^{-|x|^2/(2\eps)} ,\quad x\in \R^d.
  \end{equation*}
The function $p^\eps$ is a 
so-called \emph{coherent state} in quantum mechanics (centered at the
origin in the phase space). 
\label{sec:coherent}
\begin{lemma}
  Let $d\ge 1$. Then 
\begin{equation*}
  \|p^\eps\|_{H^s(\R^d)} \approx 
\left\{
  \begin{aligned}
    \eps^{-s/2}&\text{ if } s>-d/2,\\
    \eps^{d/4}& \text{ if } s<-d/2.
  \end{aligned}
\right.
\end{equation*}
\end{lemma}
\begin{proof}
  Resume the above computation, with now $a  = 1/ \eps$ and $b=0$. We
  have
  \begin{equation*}
    \|p^\eps\|_{H^s(\R^d)}^2 =
    \int_{\R^d}\<\frac{\eta}{\sqrt\eps}\>^{2s} 
e^{-|\eta|^2}d\eta.
  \end{equation*}
We can then resume the same computations, by simply replacing $\eps$
with $\sqrt\eps$. 
\end{proof}

\end{example}

\bibliographystyle{amsplain}
\bibliography{biblio}

\end{document}